\documentclass[oneside,american]{amsart}
\usepackage{lmodern}

\usepackage[T1]{fontenc}
\usepackage[latin9]{inputenc}
\usepackage{geometry}
\usepackage{color}
\usepackage{babel}
\usepackage{booktabs}
\usepackage{amsthm}
\usepackage{amstext}
\usepackage{amssymb}
\usepackage{esint}
\usepackage[numbers]{natbib}
\usepackage[unicode=true,
 bookmarks=true,bookmarksnumbered=false,bookmarksopen=false,
 breaklinks=false,pdfborder={0 0 0},backref=false,colorlinks=true]
 {hyperref}
\hypersetup{
 linkcolor=Dandelion, citecolor=LimeGreen, urlcolor=JungleGreen}

\makeatletter

\providecommand{\tabularnewline}{\\}

\numberwithin{equation}{section}
\numberwithin{figure}{section}
  \theoremstyle{plain}
  \newtheorem*{thm*}{\protect\theoremname}
\theoremstyle{plain}
\newtheorem{thm}{\protect\theoremname}[section]
  \theoremstyle{definition}
  \newtheorem{defn}[thm]{\protect\definitionname}
  \theoremstyle{remark}
  \newtheorem{rem}[thm]{\protect\remarkname}
  \theoremstyle{plain}
  \newtheorem{prop}[thm]{\protect\propositionname}
  \theoremstyle{plain}
  \newtheorem{lem}[thm]{\protect\lemmaname}
  \theoremstyle{plain}
  \newtheorem{cor}[thm]{\protect\corollaryname}
  \theoremstyle{definition}
  \newtheorem{example}[thm]{\protect\examplename}

\usepackage{euscript}
\usepackage[usenames, dvipsnames]{xcolor}

\makeatother

  \providecommand{\corollaryname}{Corollary}
  \providecommand{\definitionname}{Definition}
  \providecommand{\examplename}{Example}
  \providecommand{\lemmaname}{Lemma}
  \providecommand{\propositionname}{Proposition}
  \providecommand{\remarkname}{Remark}
  \providecommand{\theoremname}{Theorem}
\providecommand{\theoremname}{Theorem}

\begin{document}

\title{Splitting theorems for hypersurfaces in Lorentzian manifolds}

\author{Melanie Graf}

\address{Faculty of Mathematics, University of Vienna}

\email{melanie.graf@univie.ac.at}

\date{15.09.2016}
\begin{abstract}
The splitting problem for spacetimes with timelike Ricci curvature
bounded below by zero has been discussed extensively in the past (most
notably by Eschenburg, Galloway and Newman), in particular there
exist versions for both spacetimes containing a complete timelike
line and spacetimes containing a maximal hypersurface $\Sigma$ and
a (future) complete $\Sigma$-ray. For timelike Ricci curvature bounded
below by some $\kappa>0$ only the analogue to the first case has
been shown explicitly (see \citep{AGH1997}). 

In this paper we employ their methods (a geometric maximum principle
for the level sets of the Busemann function) to study analogues of
the second case for hypersurfaces with mean curvature bounded from
above by $\beta$. We show that given a $\Sigma$-ray of maximal length
$J^{+}(\Sigma)$ is isometric to a warped product if either $\kappa>0$
or $\beta\leq-(n-1)\sqrt{\left|\kappa\right|}$. Additionally we present
an elementary proof for such a splitting if one assumes that the volume
of (future) distance balls over subsets of this hypersurface is maximal.
\end{abstract}
\maketitle

\section{Introduction}

Over the past 50 years the study of comparison and rigidity theorems
has been an important part of Riemannian geometry and, as so often
the case, this interest soon carried over to Lorentzian geometry.
In the Riemannian context important results for manifolds with a bound
on the Ricci curvature (instead of the sectional curvatures) include
Myers's theorem, the maximal diameter theorem (\citep[Thm.~3.1]{Cheng})
and the Cheeger-Gromoll splitting theorem (\citep[Thm.~2]{CGro}),
which is already very similar to the most interesting Lorentzian case
from a physics point of view:
\begin{thm*}
[Cheeger-Gromoll splitting theorem] Let $(M,g)$ be a complete Riemannian
manifold of dimension $\geq2$ which satisfies
\[
\mathbf{Ric}(v,v)\geq0\:\:\mathrm{for\, all}\: v\in TM
\]
and which contains a complete geodesic line (i.e., a complete geodesic
that is minimizing between each of its points). Then $(M,g)$ can
be decomposed uniquely as an isometric product $N\times\mathbb{R}^{k}$,
where $N$ contains no lines, $k\geq1$ and $\mathbb{R}^{k}$ is equipped
with the standard euclidean metric.
\end{thm*}
In Lorentzian geometry one usually assumes only a bound on the timelike
Ricci curvature, i.e., we want to look at spacetimes $(M,g)$ where
\[
\mathbf{Ric}(v,v)\geq-(n-1)\kappa\, g(v,v)\:\:\mathrm{for\, all\, timelike\,}v\in TM
\]
for some $\kappa\in\mathbb{R}$.

So far most results have been focused on spacetimes having non-negative
timelike Ricci curvature (i.e., satisfying the \emph{strong energy
condition}), the exception being Andersson, Galloway and Howard (\citep{AGH1997}) who looked at the
case $\kappa>0$. A nice overview of past work can also be found in
\citep[Ch.~14]{BEE96}. 

The first Lorentzian splitting theorem for spacetimes using a bound
on the Ricci curvature instead of the sectional curvatures (for non-positive
timelike sectional curvatures the first such result was obtained by
Beem, Ehrlich, Markvorsen and Galloway in 1985, \citep{BEMG1985})
was due to Eschenburg in 1988 (\citep{eschenburg1988}) who additionally
assumed both global hyperbolicity and timelike geodesic completeness.
Shortly thereafter Galloway showed that the assumption of only global
hyperbolicity is sufficient (\citep{galloway1989_3}) and a year later
Newman gave a proof assuming timelike geodesic completeness but not
global hyperbolicity (\citep{newman1990}). These three results are
summarized as follows:
\begin{thm*}
[Lorentzian splitting theorem] Let $(M,g)$ be a spacetime of dimension
$n\geq2$ that
\begin{enumerate}
\item is either globally hyperbolic or timelike geodesically complete
\item satisfies the strong energy condition and
\item contains a complete timelike line (i.e., a curve maximizing the distance
between any of its points).
\end{enumerate}
Then $(M,g)$ splits isometrically as a product $(\mathbb{R}\times V,-dt^{2}\oplus h)$,
where $(V,h)$ is a complete Riemannian manifold.
\end{thm*}
While $\kappa=0$ certainly is the most important case from a physical
point of view, it nevertheless seems to be interesting to give a complete
description under which curvature assumptions similar results hold.
To allow for spacetimes that behave differently in one time direction
than in the other (e.g., ones that are incomplete to the future but
not to the past) we assume the existence of a smooth, acausal spacelike
hypersurface $\Sigma$ that is future causally complete (cf. Def.\
\ref{def: FCC}) and look at both a lower bound $\kappa$ on the timelike
Ricci curvature and an upper bound $\beta$ for the mean curvature
$H_{\Sigma}$ of $\Sigma$. This combination has so far only been
studied for $\kappa=\beta=0$, which again is a case of exceptional
physical interest because for $\kappa=0$ and $\beta<0$ these are
exactly the curvature assumptions in the Hawking singularity theorem.
Here \citep{GALLOWAY1989_4} showed
\begin{thm*}
Let $M$ be a space-time which obeys the strong energy condition containing
a smooth acausal maximal (i.e., zero mean curvature) spacelike hypersurface
$\Sigma$, which is either geodesically complete or future causally
complete. Assume $J^{+}(\Sigma)$ is future timelike geodesically
complete. If $\gamma$ is a future complete $\Sigma$-ray such that
$I^{-}(\gamma)\cap J^{+}(\Sigma)$ is globally hyperbolic then $J^{+}(\Sigma)$
is isometric to $([0,\infty)\times\Sigma,-dt^{2}\oplus h)$, where
$h$ is the induced metric on $\Sigma$.
\end{thm*}
For general $\kappa,\beta$ there is recent work by Treude and Grant
(\citep{TG}) using Riccati comparison theorems from \citep{EH} to
derive comparison results regarding the time evolution of the area
and volume of subsets of $\Sigma$, comparing them to the evolution
in fixed Lorentzian warped product manifolds. Similar comparison
techniques have been used in the past with the Raychaudhuri equation
to show the Hawking singularity theorem, or more precisely that no
timelike geodesic starting at $\Sigma$ can have length greater than
$-\frac{n-1}{\beta}$ if $\kappa=0$ and $\beta<0$ (see, e.g., \citep{Seno1}
for an overview). Those same techniques can be used to show that
this length is bounded from above by a constant $b_{\kappa,\beta}\leq\infty$
for arbitrary $\kappa,\beta$. Concrete values for $b_{\kappa,\beta}$
can be found in Table \ref{tab:Warping-functions-for}. Our first
goal is to investigate under which conditions the existence of an
inextendible geodesic maximizing the distance to $\Sigma$ of length
exactly $b_{\kappa,\beta}$ already implies that $I^{+}(\Sigma)$
is isometric to the warped product $(0,b_{\kappa,\beta})\times_{f_{\kappa,\beta}}(\Sigma,\frac{1}{f_{\kappa,\beta}(0)}g|_{\Sigma})$
(with $f_{\kappa,\beta}$ from Table \ref{tab:Warping-functions-for}).

For $\kappa=\beta=0$ this question is basically answered positively
by \citep[Thm.~C]{GALLOWAY1989_4} (see above), their methods relying
on the value of $b_{\kappa,\beta}$ going to infinity from below as
$\beta\nearrow0$ (and remains infinity for all $\beta\geq0$). For
$\kappa<0$ the same transition happens at $\beta=-(n-1)\sqrt{\left|\kappa\right|}$,
hence the methods used in \citep{GALLOWAY1989_4} for $\kappa=\beta=0$
would carry over to $\kappa<0,\,\beta=-(n-1)\sqrt{\left|\kappa\right|}$.
For other values $\kappa,\beta$ with $b_{\kappa,\beta}=\infty$,
i.e., $\kappa\leq0$ and $\beta>-(n-1)\sqrt{\left|\kappa\right|}$,
it is easy to see that similar results are false (see Example \ref{ex: other kappa beta do not work},
the spacetime containing an inextendible maximizing geodesic is nothing
``special'' in that case). For the remaining variations of $\kappa,\beta$
(with $b_{\kappa,\beta}<\infty$, i.e., $\kappa>0$ or $\beta<-(n-1)\sqrt{\left|\kappa\right|}$)
analogues remain true, but the proof requires a stronger (i.e., low
regularity) version of the maximum principle shown in \citep{AGH1997},
which also simplifies the proof in the second boundary case, so we
are not going to treat this case separately.

At this point one should also briefly mention recent results of Bernal
and Sánchez (\citep{bernalSanchez_globHypSplitting}), who showed
that actually any globally hyperbolic spacetime $(M,g)$ admits a
smooth time function $\mathcal{T}$ with smooth Cauchy hypersurfaces
$\Sigma_{\mathcal{T}}$ as level sets and thus splits isometrically
as $M\cong\mathbb{R}\times\Sigma$ with $g=-\beta d\mathcal{T}^{2}+h_{\mathcal{T}}$,
where $\Sigma$ is a smooth Cauchy hypersurface for $M$, $\beta:\mathbb{R}\times\Sigma\to\mathbb{R}_{+}$
is smooth and $h_{\mathcal{T}}$ is a Riemannian metric on $\Sigma_{\mathcal{T}}$.
Their work improves upon a classical topological splitting result
obtained by Geroch in 1970 (\citep{Geroch1970uw}). They refined their
arguments further to also show that given any spacelike Cauchy hypersurface
$\Sigma$ there exists a Cauchy temporal function $\mathcal{T}:M\to\mathbb{R}$
such that $\Sigma=\mathcal{T}^{-1}(0)$ (see \citep{BernalSanchez2006}).
One should note, however, that these results require neither curvature
nor any maximality assumptions and thus there is no additional information
on $\beta$ or the time evolution of $h_{\mathcal{T}}$ and the product
structure obtained this way will in general not be a warped product.

\bigskip{}

The outline of the paper is as follows. In sections \ref{sec:Definitions}
and \ref{sec:Comparison-results} we review basic definitions and
the comparison results presented in \citep{TG}. We also include a
table (Table \ref{tab:Warping-functions-for}) giving a detailed description
of the comparison spaces (introduced by \citep{TG}) that we will
use. 

In section \ref{sec:Maximal-injectivity-radius} we show that maximality
in the injectivity radius already implies that $M$ is (isometric
to) a warped product: While this seems to be a somewhat well-known
fact a detailed proof is hard to find and it ties in nicely with the
following results. 

In section \ref{sec:A-splitting-theorem for maximal ray} we use a
combination of arguments from \citep{eschenburg1988}, \citep{galloway1989_3},
\citep{GALLOWAY1989_4} and \citep{AGH1997} to show our main result,
which is that for $\kappa<0$ or $\beta\leq-(n-1)\sqrt{\left|\kappa\right|}$
the existence of an inextendible geodesic maximizing the distance
to $\Sigma$ of length exactly $b_{\kappa,\beta}$ already implies
that $I^{+}(\Sigma)\cong(0,b_{\kappa,\beta})\times_{f_{\kappa,\beta}}(\Sigma,\frac{1}{f_{\kappa,\beta}(0)^{2}}g|_{\Sigma})$
(with $f_{\kappa,\beta}$ from Table \ref{tab:Warping-functions-for}). 

Then in section \ref{sec:Maximal-volume-rigidity} we give an elementary
proof (that requires neither the Busemann function nor the maximum
principle) of the same result under the slightly stronger assumption
of maximality in certain volumes instead of the existence of a ray
of maximal length.

\subsection*{Notation}

Throughout, $M$ will always be a connected, Hausdorff and second
countable smooth manifold of dimension $n\geq2$ with a Lorentzian
metric $g$ and a time orientation. We also always assume that $(M,g)$
is globally hyperbolic. The curvature tensor of the metric is defined
with the convention $\mathbf{R}(X,Y)Z=\left(\left[\nabla_{X},\nabla_{Y}\right]-\nabla_{[X,Y]}\right)Z$
and we denote the Ricci tensor of $g$ by $\mathbf{Ric}$. Given a
spacelike, acausal hypersurface $\Sigma\subset M$ with future pointing
unit normal $\mathbf{n}_{\Sigma}$ we define the shape operator with
sign convention $\mathbf{S}_{\Sigma}=\nabla\mathbf{n}_{\Sigma}$ and
the mean curvature as $H_{\Sigma}=\mathrm{tr\,}\mathbf{S}_{\Sigma}$.

\section{\label{sec:Definitions}Definitions}

As usual we define causal (timelike) curves to be locally Lipschitz
continuous maps $\gamma:I\to M$ ($I$ being an interval) with $\dot{\gamma}\neq0$
and $g(\dot{\gamma},\dot{\gamma})\leq0$ ($<0$) a.e.\ and a causal
curve is called future (past) directed if $\dot{\gamma}$ is future
(past) pointing almost everywhere. For $p,q\in M$ we write $p\ll q$
if there is a future directed (f.d.) timelike curve from $p$ to $q$
and $p\leq q$ if either $p=q$ or there exists a f.d.\ causal curve
from $p$ to $q$ and we set
\begin{align*}
I^{+}(p): & =\left\{ q\in M:\, p\ll q\right\} \\
J^{+}(p): & =\{q\in M:\, p\leq q\}.
\end{align*}

\begin{defn}
[Signed time separation] Let $p\in M$. Then for $q\in M$ the \emph{future
time separation }to $p$ is defined by 
\begin{equation}
\tau_{p}(q):=\sup(\left\{ L(\gamma):\gamma\,\text{is a f.d. causal curve form }p\text{ to }q\right\} \cup\{0\}),\label{eq:point time sep}
\end{equation}
where $L(\gamma)$ denotes the Lorentzian arc-length of $\gamma$,
i.e., for a curve $\gamma:(t_{1},t_{2})\to M$ one has $L(\gamma):=\int_{t_{1}}^{t_{2}}\sqrt{|g(\dot{\gamma(t)},\dot{\gamma(t)})|}dt$. 

Similarly one defines the \emph{signed time separation} to an acausal
subset $\Sigma$ by
\begin{align}
\tau_{\Sigma}(p): & =\begin{cases}
\sup_{q\in\Sigma}\tau(q,p) & p\in I^{+}(\Sigma)\\
-\sup_{q\in\Sigma}\tau(p,q) & p\in I^{-}(\Sigma)\\
0 & \mathrm{otherwise}
\end{cases}.\label{eq:subset time sep}
\end{align}

\end{defn}
It is easy to see that both the time separation to a point and to
an acausal subset satisfy the reverse triangle inequality
\begin{equation}
\tau_{p}(q)+\tau_{q}(r)\leq\tau_{p}(r)\:\mathrm{and}\:\tau_{\Sigma}(q)+\tau_{q}(r)\leq\tau_{\Sigma}(r)\label{eq: RTI}
\end{equation}
for $p\leq q\leq r$ and $r\geq q\in I^{+}(\Sigma)$, respectively.

If $(M,g)$ is globally hyperbolic (i.e., it contains no closed causal
curves and $J^{+}(p)\cap J^{-}(q)$ is compact for all $p,q\in M$)
then any two points $p,q\in M$ with $p\leq q$ can be connected by
a maximizing curve (\citep[Prop.~14.19]{ONeill_SRG}). If an acausal
subset has the following property, one also gets the existence of
maximizing curves to this subset.
\begin{defn}
[Future causally complete] \label{def: FCC}A subset $\Sigma\subset M$
is called \emph{future causally complete} (\emph{FCC}) if for any
$p\in J^{+}(\Sigma)$ the set $J^{-}(p)\cap\Sigma$ has compact relative
closure in $\Sigma$.\end{defn}
\begin{rem}
Note that any smooth spacelike Cauchy hypersurface for $M$ is also
a smooth, spacelike, acausal, FCC hypersurface (see \citep[Lem.~14.40,~14.42~and~14.43]{ONeill_SRG}).
If $\Sigma$ is also past causally complete and $(M,g)$ is globally
hyperbolic then $\Sigma$ is a (smooth, spacelike) Cauchy hypersurface.
\end{rem}
The following Proposition sums up some common knowledge about the
(future) time-separation to an acausal ($p\leq q$ implies $p=q$
for any $p,q\in\Sigma$), FCC subset (see \citep[Thm.~2]{TG}).
\begin{prop}
\label{prop: time sep continuous}Let $\Sigma\subset M$ be an acausal,
FCC subset. Then the future time-separation $\tau_{\Sigma}:\, M\to\mathbb{R}$
to $\Sigma$ is finite-valued and continuous and for any $p\in J^{+}(\Sigma)\setminus\Sigma$
there exists $q\in\Sigma$ and a causal curve $\gamma$ from $q$
to $p$ with $\tau_{\Sigma}(p)=\tau(q,p)=L(\gamma)$. Any such maximizing
curve $\gamma$ has to be a (reparametrization of) a geodesic, which
is timelike for $p\in I^{+}(\Sigma)$ and null otherwise. If $\Sigma\subset M$
is, additionally, a spacelike hypersurface, then any maximizing geodesic
has to start orthogonally to $\Sigma$ (so in particular $I^{+}(\Sigma)=J^{+}(\Sigma)$).
\end{prop}
An important tool will be the normal exponential map to $\Sigma$.
\begin{defn}
[Normal exponential map] Let $\mathcal{D}^{N}\subset T\Sigma^{\perp}$
be the set of all $w\in T\Sigma^{\perp}$ such that $w\in\mathrm{dom}(\exp_{\pi(w)})$.
The normal exponential map $\exp^{N}:\mathcal{D}^{N}\to M$ to $\Sigma$
is defined by
\[
\exp^{N}(w):=\exp_{\pi(w)}(w)=\gamma_{v}(t)
\]
for $t\in\mathbb{R}$ and $v\in S^{+}N\Sigma$ such that $w=tv$.
\end{defn}
For any $v\in TM$ we denote by $\gamma_{v}$ the unique inextendible
geodesic starting at $\pi(v)$ with initial velocity $v$. For $v\in T\Sigma^{\perp}$
each $\gamma_{v}|_{[0,t]}$ maximizes the distance to $\Sigma$ for
small $t$, but it may not remain maximizing for larger $t$. We write
$S^{+}N\Sigma$ for the (future) unit normal bundle to $\Sigma$,
i.e.\, 
\[
S^{+}N\Sigma:=\left\{ v\in TM|_{\Sigma}:\, v\, f.p.,\, g(v,w)=0\,\forall w\in T_{\pi(v)}\Sigma\:\text{and}\: g(v,v)=-1\right\} \subset T\Sigma^{\perp},
\]
and define
\begin{defn}
[Cut function]The function 
\begin{align*}
s_{\Sigma}^{+}:S^{+}N\Sigma & \to[0,\infty]\\
s_{\Sigma}^{+}(v) & :=\sup\left\{ t>0:\,\tau_{\Sigma}(\gamma_{v}(t))=L(\gamma_{v}|_{\left[0,t\right]})\right\} 
\end{align*}
is called future cut function.
\end{defn}
An easy adaptation (looking at a hypersurface instead of a point)
of arguments from \citep[Prop.~9.7~and~Thm.~9.8]{BEE96} (see also
\citep[3.2.29]{Treude_Diplomarbeit}) shows
\begin{lem}
\label{lem:The-cut-function is continuous}The cut function $s_{\Sigma}^{+}$
is lower semi-continuous and continuous at points $v$ where $s_{\Sigma}^{+}(v)=\infty$
or $s_{\Sigma}^{+}(v)v\in\mathcal{D}^{N}$. \end{lem}
\begin{defn}
[Cut locus]The (future) cut locus of $\Sigma$ is defined as the
image of the tangential cut locus under the normal exponential map:
\[
\mathrm{Cut}^{+}(\Sigma):=\left\{ \exp^{N}(s_{\Sigma}^{+}(v)v):\, v\in S^{+}N\Sigma\:\:\mathrm{and}\, s_{\Sigma}^{+}(v)v\in\mathcal{D}^{N}\right\} .
\]

\end{defn}
An important fact is that $\mathrm{Cut}^{+}(\Sigma)$ has measure
zero, is closed and $\exp^{N}|_{J_{T}(\Sigma)^{\circ}}$ (where $J_{T}(\Sigma):=\left\{ tv:\, v\in S^{+}N\Sigma\:\:\mathrm{and}\, t\in[0,s_{\Sigma}^{+}(v))\right\} $)
is a diffeomorphism onto $I^{+}(\Sigma)\setminus\mathrm{Cut}^{+}(\Sigma)$
(\citep[Thm.~3]{TG}).

\section{\label{sec:Comparison-results}Comparison results}

In this section we will briefly review the comparison results from
\citep{TG}. We will generally omit the proofs, but may give a sketch
if it will be helpful later on. First, we need to define the sets
of whose areas respectively volumes will be estimated.
\begin{defn}
[Future spheres and balls] \label{def: spheres and balls}For any
$t>0$ and $A\subset\Sigma$ we define the spheres $S_{A}^{+}(t)$
and balls $B_{A}^{+}(t)$ of time $t$ above $A$ by
\begin{align*}
S_{A}^{+}(t): & =\{p\in I^{+}(\Sigma):\,\exists q\in A\,\mathrm{with}\, d(q,p)=\tau_{\Sigma}(p)=t\}\:\mathrm{and}\\
B_{A}^{+}(t): & =\bigcup_{s\in\left(0,t\right)}S_{A}^{+}(s)
\end{align*}
We also set $\mathcal{I}^{+}(\Sigma):=I^{+}(\Sigma)\setminus\mathrm{Cut}^{+}(\Sigma)$
and 
\begin{align*}
\EuScript S_{A}^{+}(t): & =S_{A}^{+}(t)\setminus\mathrm{Cut}^{+}(\Sigma)\:\:\mathrm{and}\\
\mathcal{B}_{A}^{+}(t): & =B_{A}^{+}(t)\setminus\mathrm{Cut}^{+}(\Sigma).
\end{align*}

\end{defn}
Second, we need appropriate curvature conditions.
\begin{defn}
[Cosmological comparison condition] \label{def: CCC}Let $\kappa,\beta\in\mathbb{R}$.
We say that $\left(M,g,\Sigma\right)$ satisfies the cosmological
comparison condition $CCC(\kappa,\beta)$ if
\begin{enumerate}
\item $\left(M,g\right)$ is a globally hyperbolic spacetime and $\Sigma\subset M$
is a smooth, connected, spacelike, acausal, FCC hypersurface,
\item the mean curvature $H_{\Sigma}$ of $\Sigma$ satisfies $H_{\Sigma}\leq\beta$
and
\item $\mathbf{Ric}(v,v)\geq-\left(n-1\right)\kappa\, g(v,v)$ for all timelike
$v\in TM$.
\end{enumerate}
\end{defn}
Under these assumptions \citep{TG} showed various estimates for mean
curvature, area and volume, comparing them to the respective quantities
in certain warped products $M_{\kappa,\beta}=(a_{\kappa,\beta},b_{\kappa,\beta})\times_{f_{\kappa,\beta}}\Sigma_{\kappa,\beta}$
where $a_{\kappa,\beta},b_{\kappa,\beta}\in\mathbb{R}$, $f_{\kappa,\beta}:(a_{\kappa,\beta},b_{\kappa,\beta})\mapsto\mathbb{R}\setminus\{0\}$
is the warping function and $\Sigma_{\kappa,\beta}$ is the (unique)
simply connected, complete $(n-1)$-dimensional Riemannian manifold
with constant curvature $k_{\kappa,\beta}\in\{-1,0,1\}$ (i.e., $\Sigma_{\kappa,\beta}$
is either hyperbolic space $H^{n-1}$, euclidean space $\mathbb{R}^{n-1}$
or the sphere $S^{n-1}$). These comparison spaces are listed in Table
\ref{tab:Warping-functions-for}. 
\begin{table}

\begin{centering}
\begin{tabular}{cccccc}
\multicolumn{6}{c}{Table for $\kappa<0$}\tabularnewline
\midrule
\midrule 
$\beta$ & $\Sigma_{\kappa,\beta}$ & $c$ & $f_{\kappa,\beta}(t)$ & $\frac{1}{n-1}H_{\kappa,\beta}(t)$ & $b_{\kappa,\beta}$\tabularnewline
\midrule
\midrule 
$\frac{\left|\beta\right|}{(n-1)\sqrt{\left|\kappa\right|}}<1$ & $S^{n-1}$ & $\tanh^{-1}(\frac{\beta}{(n-1)\sqrt{\left|\kappa\right|}})$ & $\frac{1}{\sqrt{\left|\kappa\right|}}\cosh(\sqrt{\left|\kappa\right|}t+c)$ & $\sqrt{\left|\kappa\right|}\tanh(\sqrt{\left|\kappa\right|}t+c)$ & $\infty$\tabularnewline
\midrule
$\frac{\left|\beta\right|}{(n-1)\sqrt{\left|\kappa\right|}}=1$ & $\mathbb{R}^{n-1}$ & $0$ & $\exp(\mathrm{sgn}(\beta)\sqrt{\left|\kappa\right|}t)$ & $\mathrm{sgn}(\beta)\sqrt{\left|\kappa\right|}$ & $\infty$\tabularnewline
\midrule
$\frac{\beta}{(n-1)\sqrt{\left|\kappa\right|}}>1$ & $H^{n-1}$ & $\coth^{-1}(\frac{\beta}{(n-1)\sqrt{\left|\kappa\right|}})$ & $\frac{1}{\sqrt{\left|\kappa\right|}}\sinh(\sqrt{\left|\kappa\right|}t+c)$ & $\sqrt{\left|\kappa\right|}\coth(\sqrt{\left|\kappa\right|}t+c)$ & $\infty$\tabularnewline
\midrule
$\frac{\beta}{(n-1)\sqrt{\left|\kappa\right|}}<-1$ & $H^{n-1}$ & $\coth^{-1}(\frac{\beta}{(n-1)\sqrt{\left|\kappa\right|}})$ & $\frac{1}{\sqrt{\left|\kappa\right|}}\sinh(\sqrt{\left|\kappa\right|}t+c)$ & $\sqrt{\left|\kappa\right|}\coth(\sqrt{\left|\kappa\right|}t+c)$ & $-\frac{c}{\sqrt{\left|\kappa\right|}}$\tabularnewline
\midrule
\midrule 
 &  &  &  &  & \tabularnewline
\multicolumn{6}{c}{Table for $\kappa=0$}\tabularnewline
\midrule
\midrule 
$\beta$ & $\Sigma_{\kappa,\beta}$ & $c$ & $f_{\kappa,\beta}(t)$ & $\frac{1}{n-1}H_{\kappa,\beta}(t)$ & $b_{\kappa,\beta}$\tabularnewline
\midrule
\midrule 
$\beta=0$ & $\mathbb{R}^{n-1}$ & $0$ & $1$ & 0 & $\infty$\tabularnewline
\midrule
$\beta>0$ & $H^{n-1}$ & $\frac{n-1}{\beta}$ & $t+c$ & $\frac{1}{t+c}$ & $\infty$\tabularnewline
\midrule
$\beta<0$ & $H^{n-1}$ & $\frac{n-1}{\beta}$ & $t+c$ & $\frac{1}{t+c}$ & $-\frac{n-1}{\beta}$\tabularnewline
\midrule
\midrule 
 &  &  &  &  & \tabularnewline
\multicolumn{6}{c}{Table for $\kappa>0$}\tabularnewline
\midrule
\midrule 
$\beta$ & $\Sigma_{\kappa,\beta}$ & $c$ & $f_{\kappa,\beta}(t)$ & $\frac{1}{n-1}H_{\kappa,\beta}(t)$ & $b_{\kappa,\beta}$\tabularnewline
\midrule
\midrule 
$\beta>0$ & $H^{n-1}$ & $\cot^{-1}(\frac{\beta}{(n-1)\sqrt{\kappa}})$ & $\frac{1}{\sqrt{\kappa}}\sin(\sqrt{\kappa}t+c)$ & $\sqrt{\kappa}\cot(\sqrt{\kappa}t+c)$ & $\frac{-c+\pi}{\sqrt{\kappa}}$\tabularnewline
\midrule 
$\beta<0$ & $H^{n-1}$ & $\cot^{-1}(\frac{\beta}{(n-1)\sqrt{\kappa}})$ & $\frac{1}{\sqrt{\kappa}}\sin(\sqrt{\kappa}t+c)$ & $\sqrt{\kappa}\cot(\sqrt{\kappa}t+c)$ & $\frac{-c}{\sqrt{\kappa}}$\tabularnewline
\midrule 
$\beta=0$ & $H^{n-1}$ & $\frac{\pi}{2}$ & $\frac{1}{\sqrt{\kappa}}\cos(\sqrt{\kappa}t)$ & $\sqrt{\kappa}\tan(\sqrt{\kappa}t)$ & $\frac{\pi}{2\sqrt{\kappa}}$\tabularnewline
\bottomrule
\end{tabular}
\par\end{centering}

\medskip{}

\caption{\label{tab:Warping-functions-for}Warping functions for different
values of $\kappa,\beta$. The mean curvature is given by $H_{\kappa,\beta}=(n-1)\frac{f'}{f}$
and $b_{\kappa,\beta}$ is the upper bound of the interval containing
zero on which $f_{\kappa,\beta}\neq0$ and $c=c(\kappa,\beta)$ is
a ($\kappa,\beta$ dependent) constant. This table is based on \citep[Table~1]{TG}.}
\end{table}

Now we are ready to state some of the relevant results of \citep{TG}:
Their arguments are based on a comparison result for Riccati equations
from \citep{EH}, which they then adapt to their needs \citep[Thm.\ 6]{TG}.
\begin{thm}
[Riccati comparison] \label{thm: riccati comp}Let $R:\mathbb{R}\to S(E)$
(self-adjoint operators from an $n$-dimensional vector space $E$
into itself) be smooth and assume that $\mathrm{tr}\, R\geq n\kappa$
for some $\kappa\in\mathbb{R}$. Furthermore, let $S:(0,b)\to S(E)$
be a solution of $S'+S^{2}+R=0$, and $s_{\kappa}:(0,b_{\kappa})\to\mathbb{R}$
a solution of $s_{\kappa}'+s_{\kappa}^{2}+\kappa=0$ that can not
be extended beyond $b_{\kappa}$. If $\lim_{t\searrow0}(s_{\kappa}(t)-\frac{1}{n}\mathrm{tr}S)$
exists and is non-negative, then $b\leq b_{\kappa}$ and 
\[
\mathrm{tr}\, S(t)\leq ns_{\kappa}(t)
\]
for all $t\in(0,b)$. Moreover, if equality holds for some $t_{0}\in(0,b)$,
then equality holds for all $t<t_{0}$. In this case, we also have
$S(t)=s_{\kappa}(t)\mathrm{Id}_{E}$ and $R(t)=\kappa\mathrm{Id}_{E}$
for all $t\in(0,t_{0}]$.
\end{thm}
It is easy to see that the shape operators $\mathbf{S}_{t}=-\nabla\mathrm{grad}\tau_{\Sigma}$
to the level sets $\Sigma_{t}:=\tau_{\Sigma}^{-1}(\{t\})$ satisfy
such a Riccati equation along each timelike geodesic $\gamma$ starting
orthogonally to $\Sigma$ with $R_{t}:\dot{\gamma}(t)^{\perp}\to\dot{\gamma}(t)^{\perp}$
given by $\mathbf{R}(.,\dot{\gamma}(t))\dot{\gamma}(t)$. This leads
to the following result about the mean curvatures $H_{t}:=H_{\Sigma_{t}}$
of the level sets $\Sigma_{t}$ \citep[Thm.\ 7]{TG}:
\begin{thm}
[Mean curvature comparison] \label{thm:mean curv comp}Let $\kappa,\beta\in\mathbb{R}$
and assume that $M$ and $\Sigma\subset M$ satisfy $CCC(\kappa,\beta)$.
Then
\begin{enumerate}
\item For any inextendible unit-speed geodesic $\gamma:[0,a)\to M$ maximizing
the distance to $\Sigma$, one has $H_{t}(\gamma(t))\leq H_{\kappa,H_{\Sigma}(\gamma(0))}(t)$
along $\gamma$, hence $a<b_{\kappa,H_{\Sigma}(\gamma(0))}$
\item For each $q\in\mathcal{I}^{+}(\Sigma)$, we have $\tau_{\Sigma}(q)<b_{\kappa,\beta}$
and $\mathrm{tr}\,\mathbf{S}_{\tau_{\Sigma}(q)}(q)=H_{\tau_{\Sigma}(q)}(q)\leq H_{\kappa,\beta}(\tau_{\Sigma}(q))=\mathrm{tr}\,\mathbf{S}_{\kappa,\beta}(\tau_{\Sigma}(q)).$ 
\item If $H_{\tau_{\Sigma}(q)}(q)=H_{\kappa,\beta}(\tau_{\Sigma}(q))$ and
$\gamma:[0,\tau_{\Sigma}(q)]\to M$ is the (unique unit-speed) geodesic
maximizing the distance from $q$ to $\Sigma$, then even $\mathbf{S}_{t}(\gamma(t))=\frac{1}{n-1}H_{\kappa,\beta}(t)\,\mathrm{id}$
for all $t\in[0,\tau_{\Sigma}(q)]$.
\end{enumerate}
\end{thm}
Not stated explicitly in \citep{TG} is an immediate corollary we
will use later on.
\begin{cor}
\label{cor: tau<b on I+}Actually $\tau_{\Sigma}(q)<b_{\kappa,\beta}$
for all $q\in I^{+}(\Sigma)$.\end{cor}
\begin{proof}
From $\tau_{\Sigma}(q)<b_{\kappa,\beta}$ for any $q\in\mathcal{I}^{+}(\Sigma)$
it follows from density of $\mathcal{I}^{+}(\Sigma)$ in $I^{+}(\Sigma)$
that $\tau_{\Sigma}(q)\leq b_{\kappa,\beta}$ for any $q\in I^{+}(\Sigma)$.
Now assume there exists a $q\in I^{+}(\Sigma)$ with$\tau_{\Sigma}(q)=b_{\kappa,\beta}$
and let $\gamma:[0,b_{\kappa,\beta}]\to M$ be a geodesic maximizing
the distance from $\Sigma$ to $q$. By extending this geodesic we
get a point $q'\in I^{+}(\Sigma)$ with $\tau_{\Sigma}(q')>b_{\kappa,\beta}$,
arriving at a contradiction.
\end{proof}
Next \citep{TG} use a standard result on the variation of area (see
\citep[Ch.~2]{Simon_GeomMeasure}).
\begin{prop}
[First variation of area] \label{prop:first var of area}Let $K\subset\EuScript S_{t}$
be compact and let $\varepsilon>0$ be such that the flow, $\Phi$,
of $\mathbf{n}$ is defined on $[-\varepsilon,\varepsilon]\times K$.
Set $K_{s}:=\Phi_{s}(K)\subset\EuScript S_{t+s}$ for each $s\in[-\varepsilon,\varepsilon]$.
Then
\begin{equation}
\left.\frac{d}{ds}\right|_{s=0}\mathrm{area}K_{s}=\int_{K}\mathrm{tr}\,\mathbf{S}_{t}d\mu_{t},\label{eq:var of area}
\end{equation}
where $\mu_{t}$ denotes the Riemannian volume measure on $\EuScript S_{t}$
induced by $g$.
\end{prop}
This allows them to proof the following area comparison theorem \citep[Thm.~8]{TG}.
\begin{thm}
[Area comparison]\label{thm:area comp} Let $\kappa,\beta\in\mathbb{R}$
and assume that $\left(M,g,\Sigma\right)$ satisfies $CCC(\kappa,\beta)$.
Then, for any  $B\subset\Sigma_{\kappa,\beta}$ with finite, non-zero
measure and any measurable $A\subset\Sigma$, the function 
\[
t\mapsto\frac{\mathrm{area}\,\EuScript S_{A}^{+}(t)}{\mathrm{area}_{\kappa,\beta}S_{B}^{+}(t)},
\]
is nonincreasing on $[0,b_{\kappa,\beta})$. Further, for $t\searrow0$
this function converges to $\frac{\mathrm{area}A}{\mathrm{area}_{\kappa,\beta}B}$,
so
\[
\mathrm{area}\,\EuScript S_{A}^{+}(t)\leq\frac{\mathrm{area}_{\kappa,\beta}S_{B}^{+}(t)}{\mathrm{area}_{\kappa,\beta}B}\,\mathrm{area}A.
\]
\end{thm}
\begin{proof}
We will only give a sketch here. If $A$ is compact and the flow $\Phi$
of the unit normal vector field $\mathbf{n}$ is defined on $[0,b_{\kappa,\beta})$,
then $\EuScript S_{A}^{+}(t)=\Phi_{t}(A)$ and one can use Prop.\
\ref{prop:first var of area} and Thm.\ \ref{thm:mean curv comp}
to calculate
\begin{equation}
\frac{d}{dt}\log(\mathrm{area}\,\EuScript S_{A}^{+}(t))=\frac{1}{\mathrm{area}\,\EuScript S_{A}^{+}(t)}\int_{\EuScript S_{A}^{+}(t)}H_{t}(q)d\mu_{t}(q)\leq H_{\kappa,\beta}(t)=\frac{d}{dt}\log(\mathrm{area}_{\kappa,\beta}S_{B}^{+}(t)),\label{eq:log area}
\end{equation}
proving the assertion. If this is not the case, one looks at $0<t_{1}<t_{2}<b_{\kappa,\beta}$
and a sequence of compact sets $K_{i}\subset\EuScript S_{A}^{+}(t_{2})$
with $\mathrm{area}K_{i}\nearrow\mathrm{area}\,\EuScript S_{A}^{+}(t_{2})$
and uses the sets $K_{i}(t):=\Phi_{t-t_{2}}(K_{i})$ (for $t\in[0,t_{2}]$)
instead of $\EuScript S_{A}^{+}(t)$ in (\ref{eq:log area}).
\end{proof}
The co-area formula (note that $\mathrm{Cut}^{+}(\Sigma)$ has measure
zero) implies
\begin{equation}
\mathrm{vol}B_{A}^{+}(t)=\int_{0}^{t}\mathrm{area}\EuScript S_{A}^{+}(\tau)d\tau,\label{eq:coarea formula}
\end{equation}
and some basic analysis regarding integrals of functions with a non-increasing
quotient gives \citep[Thm.~9]{TG}:
\begin{thm}
[Volume comparison]\label{thm:vol comp}Let $\kappa,\beta\in\mathbb{R}$
and assume $\left(M,g,\Sigma\right)$ satisfies $CCC(\kappa,\beta)$.
Then, for any  $B\subset\Sigma_{\kappa,\beta}$ with finite, non-zero
measure and any measurable $A\subset\Sigma$, the function 
\[
t\mapsto\frac{\mathrm{vol}B_{A}^{+}(t)}{\mathrm{vol}_{\kappa,\beta}B_{B}^{+}(t)},
\]
is nonincreasing. Further, for $t\searrow0$ this function converges
to $\frac{\mathrm{area}A}{\mathrm{area}_{\kappa,\beta}B}$, so
\begin{equation}
\mathrm{vol}B_{A}^{+}(t)\leq\frac{\mathrm{vol}_{\kappa,\beta}B_{B}^{+}(t)}{\mathrm{area}_{\kappa,\beta}B}\,\mathrm{area}A.\label{eq:vol leq sth}
\end{equation}

\end{thm}
A similar result has also recently been shown for $\mathcal{C}^{1,1}$-metrics
(\citep{G2016}).

Moving away from the hypersurface case for a moment we will also need
a comparison theorem for the d'Alembertian of the distance function
to a point. This seems to be a well known result (see, e.g., \citep[Eq.~(14.29)]{BEE96}
for $\kappa=0$, the proof of \citep[Prop.~4.9]{AGH1997} for $\kappa<0$
or \citep[Thm.~3.3.5]{Treude_Diplomarbeit}). 
\begin{thm}
\label{thm:box of distance funct}Assume $M$ is globally hyperbolic
and its timelike Ricci curvature is bounded from below by $\kappa\in\mathbb{R}$.
Fix $p\in M$. Then for any $q\in I^{+}(p)\setminus\mathrm{Cut}^{+}(p)$
we have
\begin{equation}
-\Box\tau_{p}(q)\leq(n-1)s_{\kappa}(\tau_{p}(q)),\label{eq:box of distance}
\end{equation}
where
\begin{equation}
s_{\kappa}(t):=\begin{cases}
\sqrt{\kappa}\cot(\sqrt{\kappa}t) & \kappa>0\\
\frac{1}{t} & \kappa=0\\
\sqrt{\left|\kappa\right|}\coth(\sqrt{\left|\kappa\right|}t) & \kappa<0
\end{cases}.\label{eq:def s kappa}
\end{equation}
\end{thm}
\begin{proof}
As in Thm.\ \ref{thm:mean curv comp} we have that along a maximizing,
unit speed geodesic $\gamma$ from $p$ to $q$ the function $f(t):=-\Box\tau_{p}(\gamma(t))=\mathrm{tr}\,\mathbf{S}_{\tau_{p}^{-1}(t)}(\gamma(t))$
is smooth ($q\notin\mathrm{Cut}^{+}(p)$) and satisfies
\[
f'+\frac{f^{2}}{2}\leq(n-1)\kappa\;\mathrm{and}\;(s_{\kappa}(t)-\frac{1}{n-1}f(t))\to0\,\mathrm{as}\, t\searrow0,
\]
where the limiting behavior is seen by looking at Minkowski space.
This gives (\ref{eq:box of distance}) by Thm.\ \ref{thm: riccati comp}.
\end{proof}

\section{\label{sec:Maximal-injectivity-radius}Maximality in the injectivity
radius}

In the next three sections we will investigate manifolds $(M,g,\Sigma$)
satisfying $CCC(\kappa,\beta)$ which are in a sense maximal with
respect to the bounds on distance, area and volume from Thm.\ \ref{thm:mean curv comp}-\ref{thm:vol comp}
implied by the curvature. 

The first (and simplest) involves maximality in the $\Sigma$-injectivity
radius of $M$ and although this seems to be a somewhat well-known
fact, we will nevertheless provide a detailed proof.
\begin{defn}
The future $\Sigma$-injectivity radius $\mathrm{inj}_{\Sigma}^{+}(M)$
is defined as the infimum over the future cut parameter of points
in $\Sigma$, i.e., 
\[
\mathrm{inj}_{\Sigma}^{+}(M):=\inf_{p\in\Sigma}s_{\Sigma}^{+}(p).
\]

\end{defn}
Note that $\exp^{N}|_{(0,\mathrm{inj}_{\Sigma}^{+}(M))\cdot S^{+}N\Sigma}$
will be a diffeomorphism onto $B_{\Sigma}^{+}(\mathrm{inj}_{\Sigma}^{+}(M))\setminus\mathrm{Cut}^{+}(\Sigma)$.

If $(M,g,\Sigma$) satisfies $CCC(\kappa,\beta)$ for some $\kappa,\beta$,
then Cor.\ \ref{cor: tau<b on I+} shows that $\tau_{\Sigma}(q)<b_{\kappa,\beta}$
for all $q\in I^{+}(\Sigma)$, which in turn implies $\mathrm{inj}_{\Sigma}^{+}(M)\leq b_{\kappa,\beta}$.
We will now show
\begin{thm}
[Maximal injectivity radius rigidity] \label{thm: max injectivity radius}Let
$(M,g)$ be globally hyperbolic and assume that $\left(M,g,\Sigma\right)$
satisfies $CCC(\kappa,\beta)$ with constants $\kappa,\beta$ such
that either $\kappa>0$ or $\beta\leq-(n-1)\sqrt{\left|\kappa\right|}$.
If $\mathrm{inj}_{\Sigma}^{+}(M)=b_{\kappa,\beta}$, then $I^{+}(\Sigma)$
is isometric to the warped product 
\begin{equation}
I^{+}(\Sigma)\cong(0,b_{\kappa,\beta})\times_{f_{\kappa,\beta}}(\Sigma,\frac{1}{f_{\kappa,\beta}(0)^{2}}\, g|_{\Sigma}).\label{eq:I+ isom max inj rad}
\end{equation}
\end{thm}
\begin{proof}
Since $\mathrm{inj}_{\Sigma}^{+}(M)=b_{\kappa,\beta}$ and $\tau_{\Sigma}(p)<b_{\kappa,\beta}$
for all $p\in I^{+}(\Sigma)$ one has $\mathrm{Cut}^{+}(\Sigma)=\emptyset$
and hence the normal exponential map is a diffeomorphism $\exp^{N}:\,(0,b_{\kappa,\beta})\cdot S^{+}N\Sigma\to I^{+}(\Sigma)$.
Identifying $(0,b_{\kappa,\beta})\cdot S^{+}N\Sigma$ with $(0,b_{\kappa,\beta})\times\Sigma$
in the usual way and pulling back the metric $g$ on $I^{+}(\Sigma)$
we obtain a metric $\bar{g}$ on $(0,b_{\kappa,\beta})\times\Sigma$
that is of the form $\bar{g}=-dt^{2}+h(t,x)$, where $h(t,.)$ denotes
the induced Riemannian metric on the time slice $\{t\}\times\Sigma$.
It remains to show that $\bar{g}=-dt^{2}+\frac{f_{\kappa,\beta}(t)^{2}}{f_{\kappa,\beta}(0)^{2}}\, h_{ij}(0,x)$.

Next we show that $S_{t}(q)=\frac{f'_{\kappa,\beta}(t)}{f_{\kappa,\beta}(t)}\,\mathrm{id}$
for all $t<b_{\kappa,\beta}$ and $q\in S_{\Sigma}^{+}(t)=\EuScript S_{\Sigma}^{+}(t)$.
From Thm.\ \ref{thm:mean curv comp} we know that it suffices to
show that $\frac{1}{n-1}\, H_{t}(q)=\frac{1}{n-1}H_{\kappa,\beta}(t)=\frac{f'_{\kappa,\beta}(t)}{f_{\kappa,\beta}(t)}$.
Assume to the contrary that there exists $q_{0}\in S_{\Sigma}^{+}(t_{0})$
with $\tilde{\beta}:=H_{t}(\gamma(t_{0}))<H_{\kappa,\beta}(t_{0})=:\beta_{t_{0}}$
and let $\gamma$ be the unique geodesic $\gamma$ starting orthogonally
to $\Sigma$ with $\gamma(t_{0})=q_{0}$ (any such curve maximizes
the distance due to $\mathrm{Cut}^{+}(\Sigma)=\emptyset$). Then starting
the Riccati comparison argument not at $\gamma(0)$ but at $\gamma(t_{0})$
(note that $\Sigma_{t_{0}}$ is again a smooth, acausal, spacelike,
FCC hypersurface) we see that $H_{\Sigma_{t}}(\gamma(t-t_{0}))\leq H_{\kappa,\tilde{\beta}}(t-t_{0})$
for $t>t_{0}$. Looking at table \ref{tab:Warping-functions-for}
(or Thm.\ \ref{thm: riccati comp} and \ref{thm:mean curv comp})
we see that $H_{\kappa,\tilde{\beta}}(t-t_{0})\to-\infty$ for $t-t_{0}\nearrow b_{\kappa,\tilde{\beta}}$
and that the map $\beta\to b_{\kappa,\beta}$ is strictly increasing
on $\mathbb{R}$ for $\kappa<0$ and on $(-\infty,-(n-1)\sqrt{\left|\kappa\right|}]$
for $\kappa\leq0$, hence in all cases we are considering one has
$b_{\kappa,\tilde{\beta}}<b_{\kappa,\beta_{t_{0}}}$. Using that $f_{\kappa,\beta_{t_{0}}}(t-t_{0})=f_{\kappa,\beta}(t)$
by uniqueness of solutions of ODE we see that $b_{\kappa,\beta_{t_{0}}}=b_{\kappa,\beta}-t_{0}$.
This gives $b_{\kappa,\tilde{\beta}}<b_{\kappa,\beta}-t_{0}$, i.e.,
$H_{t}(\gamma(t))\to-\infty$ for $t\nearrow b_{\kappa,\tilde{\beta}}+t_{0}<b_{\kappa,\beta}$
, which contradicts $\gamma$ not having a focal point before $b_{\kappa,\beta}$.

Now (\ref{eq:I+ isom max inj rad}) follows from 
\begin{multline}
\frac{d}{dt}h_{ij}(t,x)=\frac{d}{dt}g(\partial_{x_{i}},\partial_{x_{j}})=\nabla_{\partial_{t}}(g(\partial_{x_{i}},\partial_{x_{j}}))=g(\nabla_{\partial_{t}}\partial_{x_{i}},\partial_{x_{j}})+g(\partial_{x_{i}},\nabla_{\partial_{t}}\partial_{x_{j}})=\\
=g(\nabla_{\partial_{x_{i}}}\partial_{t},\partial_{x_{j}})+g(\partial_{x_{i}},\nabla_{\partial_{x_{j}}}\partial_{t})=g(S_{i}^{k}(t,x)\partial_{x_{k}},\partial_{x_{j}})+g(\partial_{x_{i}},S_{j}^{k}(t,x)\partial_{x_{k}})=2\frac{f'_{\kappa,\beta}(t)}{f_{\kappa,\beta}(t)}h_{ij}(t,x),\label{eq:curv calc}
\end{multline}
as the solution of this equation is given by $h_{ij}(t,x)=\frac{h_{ij}(0,x)}{f_{\kappa,\beta}(0)^{2}}\, f_{\kappa,\beta}(t)^{2}$.\end{proof}
\begin{rem}
As mentioned above this result in itself is not surprising. One can
find a related result in \citep[Thm.~5.3]{AnderssonHoward} and similar
calculations also appear in \citep{eschenburg1988}. In general, if
$\Sigma$ is a spacelike hypersurface in $M$ the normal exponential
map is defined on $(0,\mathrm{inj}_{\Sigma}^{+}(M))\cdot S^{+}N\Sigma$,
$\exp^{N}((0,\mathrm{inj}_{\Sigma}^{+}(M))\cdot S^{+}N\Sigma)\cong(0,\mathrm{inj}_{\Sigma}^{+}(M))\times\Sigma$
and the induced metric on $(0,\mathrm{inj}_{\Sigma}^{+}(M))\times\Sigma$
is adapted to this product structure (as defined by \citep[Def.~5.1]{AnderssonHoward}):
For any $p\in\Sigma$ the curve $t\mapsto c_{p}(t):=(t,p)\cong\exp^{N}(t\mathbf{n}_{p})$
is a unit speed geodesic which shows that $g$ is locally of the form
$-dt^{2}+\sum_{i,j=1}^{n-1}g_{ij}(t,x)dx_{i}dx_{j}$. So in the case
of a maximal $\Sigma$-injectivity radius, while it remains to actually
calculate the $g_{ij}(t,x)$, one gets an ``almost'' warped product
for free. This will no longer be the case if one looks at maximality
in the volume as will be done in section \ref{sec:Maximal-volume-rigidity}.\end{rem}
\begin{example}
\label{ex: other kappa beta do not work}For $\kappa\geq0$ and $\beta>-(n-1)\sqrt{\left|\kappa\right|}$
the analogue of Thm.\ \ref{thm: max injectivity radius} is false:
Obviously the warped product spacetimes $M_{\kappa,\tilde{\beta}}:=(a_{\kappa,\beta},\infty)\times_{f_{\kappa,\tilde{\beta}}}(\Sigma,\frac{1}{f_{\kappa,\tilde{\beta}}(0)}h)$
with $\tilde{\beta}\in[\beta,-(n-1)\sqrt{\left|\kappa\right|})$ satisfy
$CCC(\kappa,\beta)$, $\mathrm{inj}_{\Sigma}^{+}(M_{\kappa,\tilde{\beta}})=\infty$
and $g_{\kappa,\tilde{\beta}}|_{\Sigma}=h$ but they are not isometric
to $M_{\kappa,\beta}$ unless $\tilde{\beta}=\beta$.
\end{example}

\section{\label{sec:A-splitting-theorem for maximal ray}A splitting theorem
for hypersurfaces with a maximal ray}

The goal of this section is to show that one does not need $\mathrm{inj}_{\Sigma}^{+}(M)=b_{\kappa,\beta}$
to obtain a splitting result and that indeed the existence of only
one $\Sigma$-ray of length $b_{\kappa,\beta}$ is sufficient. As
mentioned in the introduction the proof will be a rather straightforward
combination of arguments from \citep{eschenburg1988}, \citep{galloway1989_3},
\citep{GALLOWAY1989_4} and \citep{AGH1997}. 
\begin{defn}
[(Maximal length) $\Sigma$-rays] Let $\Sigma\subset M$ be an acausal
subset. A timelike future inextendible unit-speed geodesic $\gamma:[0,a)\to M$
is called a $\Sigma$-ray if $\gamma(0)\in\Sigma$ and $\gamma$ maximizes
distance to $\Sigma$, i.e., $L(\gamma|_{[0,t]})=\tau_{\Sigma}(\gamma(t))$
for all $t\in[0,a)$. If $(M,g,\Sigma)$ satisfies $CCC(\kappa,\beta)$
we say a $\Sigma$-ray $\gamma$ has \emph{maximal length }if $a=b_{\kappa,\beta}$.
\end{defn}
To any ray one can define asymptotes:
\begin{defn}
[Asymptotes] For $p\in M$ we call an inextendible geodesic $\alpha_{p}:[0,\bar{a})\to M$
an asymptote to the ($\Sigma$-)ray $\gamma:[0,a)\to M$ at $p$ if
$\alpha_{p}(0)=p$ and $\dot{\alpha}_{p}(0)=\lim_{n\to\infty}\dot{\alpha}_{p,s_{n}}(0)$
for some sequence $s_{n}\to a$, where $\alpha_{p,s}$ denotes the
maximizing unit speed geodesic from $p$ to $\gamma(s)$. So $\alpha_{p}$
arises as a limit curve of a sequence $\alpha_{p,s_{n}}$ of maximizing
curves from $p$ to $\gamma(s_{n})$ as $s_{n}\to a$.
\end{defn}
Given a ray $\gamma:[0,a)\to M$ we define the Busemann function
$b$ associated to this ray. 
\begin{defn}
[Busemann function] Given a $\Sigma$-ray $\gamma:[0,a)\to M$ one
defines its \emph{Busemann function} $b$ as the limit
\begin{equation}
b(x):=\lim_{r\to a}r-\tau_{x}(\gamma(r))\label{eq:def busem}
\end{equation}
for $x\in I^{-}(\gamma)\cap I^{+}(\Sigma)$. \end{defn}
\begin{rem}
That this limit actually exists is seen as follows: By the reverse
triangle inequality (\ref{eq: RTI}) one has $\tau_{x}(\gamma(r))\geq\tau_{x}(\gamma(s))+r-s$
for $r\geq s\geq r_{0}$ with $r_{0}$ such that $x\in I^{-}(\gamma(r_{0}))$
so the map $r\mapsto r-\tau_{x}(\gamma(r))$ is monotonously decreasing
and using $\tau_{x}(\gamma(r))\leq\tau_{\gamma(0)}(\gamma(r))-\tau_{\Sigma}(x)$
it is easy to see that $r-\tau_{x}(\gamma(r))\geq\tau_{\Sigma}(x)$
for all $x\in I^{-}(\gamma)\cap I^{+}(\Sigma)$. This also shows 
\begin{equation}
b(x)\geq\tau_{\Sigma}(x).\label{eq:busem geq dist}
\end{equation}

\end{rem}
Before we summarize the most important facts about the Busemann function
in the following Proposition we need one more definition.

We say that a set $N\subset M$ in a spacetime $(M,g)$ is \emph{edgeless
}if for all $p\in N$ and all neighborhoods $V$ of $p$ in $M$ any
timelike curve from $I^{-}(p,V)$ to $I^{+}(p,V)$ must meet $N$.
 The following definition was introduced in \citep{EG1992}.
\begin{defn}
[$\mathcal{C}^0$ spacelike hypersurface] A subset $N\subset M$ of
a spacetime $(M,g)$ is called \emph{$\mathcal{C}^{0}$ spacelike
hypersurface} if for each $p\in N$ there is a neighborhood $U$ of
$p$ in $M$ such that $N\cap U$ is acausal and edgeless in $U$.
Note that this implies that $N$ is a topological hypersurface by
\citep[Prop.~14.25]{ONeill_SRG}.\end{defn}
\begin{prop}
[Properties of the Busemann function] \label{prop:properties of buseman}Let
$(M,g)$ be a globally hyperbolic spacetime, $\Sigma\subset M$ an
acausal, FCC, spacelike hypersurface and $\gamma:[0,a)\to M$ a $\Sigma$-ray.
Then for any $t\in(0,a)$ there is a neighborhood $U$ of $\gamma(t)$
(called a \emph{nice neighborhood}) such that the following holds:
\begin{enumerate}
\item The Busemann function $b$ is continuous on $U$ and if $q\in J^{+}(p)$
one has
\begin{equation}
b(q)\geq b(p)+\tau_{p}(q)\label{eq:b increasing}
\end{equation}

\item For any given Riemannian background metric $h$ there exists a constants
$C$ and $t<T<a$ such that for any maximizing geodesic $\alpha_{p,s}$
from a point $p\in U$ to $\gamma(s)$ with $s\geq T$ 
\begin{equation}
h(\dot{\alpha}_{p,s}(0),\dot{\alpha}_{p,s}(0))\leq C,\label{eq:blubbbbbbbb}
\end{equation}
i.e., the set $\{\dot{\alpha}_{p,s}(0):p\in U,T\leq s<a\}\subset TM$
is contained in a compact set.
\item For any $p\in U$ there exists a timelike, unit-speed asymptote $\alpha_{p}:[0,a-b(p))\to M$
at $p$ that is future inextendible, maximizing, and satisfies
\begin{equation}
b(\alpha_{p}(t))=t+b(p).\label{eq:b along asymptote}
\end{equation}

\item The level set $N_{t}:=\{x\in U:b(x)=t\}$ of $b$ in $U$ is edgeless
and acausal, i.e., a $\mathcal{C}^{0}$ spacelike hypersurface in
$U$.
\end{enumerate}
\end{prop}
\begin{proof}
This was shown in \citep{eschenburg1988}, \citep{galloway1989_3}
and \citep{GALLOWAY1989_4}, see specifically \citep[Lem.~3.3]{eschenburg1988}
for Lipschitz continuity, \citep[Lem.~3.2]{eschenburg1988} for the
estimate (\ref{eq:blubbbbbbbb}) and \citep[Lem.~2.3]{galloway1989_3}
for the properties of $N_{t}$. The existence of timelike, unit speed
asymptotes follows from (\ref{eq:blubbbbbbbb}). By a standard result
about the length functional regarding limits of curves contained in
a common compact set one then has$\limsup_{s\to a}L(\alpha_{p,s}|_{[0,t]})\leq L(\alpha_{p}|_{[0,t]})$
(for any $t>0$ such that there exists $s_{0}$ such that $\alpha_{p,s}$
is well defined on $[0,t]$ for $s_{0}<s<a$). This shows that the
asymptote is maximzing and has length at least $\limsup_{s\to a}L(\alpha_{p,s})=\lim_{s\to a}\tau_{p}(\gamma(s))=\lim_{s\to a}\left(s-(s-\tau_{p}(\gamma(s)))\right)=a-b(p)$.
Finally, because $\gamma(s)\to\infty$ (i.e., leaves every compact
set) for $s\to a$ the asymptote $\alpha_{p}:[0,a-b(p))\to M$ is
inextendible. Equations (\ref{eq:b increasing}) and (\ref{eq:b along asymptote})
are immediate consequences of the reverse triangle inequality (\ref{eq: RTI}). 

The statement is also included in \citep{AGH1997}. Note that all
of this is independent of any curvature assumptions. 
\end{proof}
The main argument we use from \citep{AGH1997} will be a theorem about
$\mathcal{C}^{0}$ spacelike hypersurfaces with curvature bounds.
Given two $\mathcal{C}^{0}$ spacelike hypersurfaces in $(M,g)$ which
meet at a point $q$ we say that $N_{0}$ is locally to the future
of $N_{1}$ near $q$ if they meet at $q$ and for some neighborhood
$U$ of $q$ in which $N_{1}$ is acausal and edgeless, $N_{0}\cap U\subset J^{+}(N_{1},U)$.
Now one can define mean curvature bounds of such a $\mathcal{C}^{0}$
spacelike hypersurface as follows
\begin{defn}
Let $N$ be a $\mathcal{C}^{0}$ spacelike hypersurface in the spacetime
$(M,g)$ and $H_{0}$ a constant. Then
\begin{enumerate}
\item $N$ has \emph{mean curvature $\leq H_{0}$ in the sense of support
hypersurfaces} if for all $q\in N$ and $\varepsilon>0$ there is
a $\mathcal{C}^{2}$ future support hypersurface $S_{q,\varepsilon}$
(i.e., $q\in S_{q,\varepsilon}$ and $S_{q,\varepsilon}$ is locally
to the future of $N$ near $q$) such that
\[
H_{S_{q,\varepsilon}}(q)\leq H_{0}+\varepsilon.
\]

\item $N$ has \emph{mean curvature $\geq H_{0}$ in the sense of support
hypersurfaces with one-sided Hessian bounds} if for all compact sets
$K\subset N$ there exists a compact set $\hat{K}\subset TM$ and
a constant $C>0$ such that for all $q\in K$ there is a $\mathcal{C}^{2}$
past support hypersurface $P_{q,\varepsilon}$ (i.e., $q\in P_{q,\varepsilon}$
and $P_{q,\varepsilon}$ is locally to the past of $N$ near $q$)
such that the future pointing unit normal $\mathbf{n}_{P_{q,\varepsilon}}(q)$
is in $\hat{K}$, the second fundamental form $h_{P_{q,\varepsilon}}$
satisfies 
\begin{equation}
h_{P_{q,\varepsilon}}(q)\geq-C_{K}g|_{P_{q,\varepsilon}}(q)\label{eq:sec fund form hessian bounds}
\end{equation}
and
\[
H_{P_{q,\varepsilon}}(q)\geq H_{0}-\varepsilon.
\]

\end{enumerate}
\end{defn}
This definition was introduced in \citep{AGH1997} and allows them
to prove a Lorentzian geometric maximum principle for $\mathcal{C}^{0}$
spacelike hypersurfaces. 
\begin{thm}
[Lorentzian Geometric Maximum Principle] \label{thm:geom max princ}Let
$N_{0}$ and $N_{1}$ be $\mathcal{C}^{0}$ spacelike hypersurfaces
in a spacetime $(M,g)$ which meet at a point $q_{0}$, such that
$N_{0}$ is locally to the future of $N_{1}$ near $q_{0}$. Assume
for some constant $H_{0}$:
\begin{enumerate}
\item $N_{0}$ has mean curvature $\leq H_{0}$ in the sense of support
hypersurfaces and
\item $N_{1}$ has mean curvature $\geq H_{0}$ in the sense of support
hypersurfaces with one-sided Hessian bounds.
\end{enumerate}
Then $N_{0}=N_{1}$ near $q_{0}$, i.e., there is a neighborhood $U$
of $q_{0}$ such that $N_{0}\cap U=N_{1}\cap U$. Moreover, $N_{0}\cap U=N_{1}\cap U$
is a smooth spacelike hypersurface with mean curvature $H_{0}$.\end{thm}
\begin{proof}
See \citep[Thm.~3.6]{AGH1997}.
\end{proof}
We are now going to show the analogue to \citep[Prop. 4.9, 4.]{AGH1997}
for our situation.
\begin{prop}
Let $(M,g)$ be a globally hyperbolic spacetime, $\Sigma\subset M$
an acausal, FCC and spacelike hypersurface, $\gamma:[0,a)\to M$ a
$\Sigma$-ray and let $U$ be a nice neighborhood of $\gamma(t)$.
If $\mathbf{Ric}(v,v)\geq-\left(n-1\right)\kappa\, g(v,v)$ for all
timelike $v\in TM$, then
\begin{equation}
H_{N_{t}}\geq-(n-1)s_{\kappa}(a-t)=\begin{cases}
-\sqrt{\kappa}\cot(\sqrt{\kappa}(a-t)) & \kappa>0\\
-\frac{1}{a-t} & \kappa=0\\
-\sqrt{\left|\kappa\right|}\coth(\sqrt{\left|\kappa\right|}(a-t)) & \kappa<0
\end{cases}\label{eq:lower bound on H N}
\end{equation}
in the sense of support hypersurfaces with one-sided Hessian bounds.
Note that by Thm.\ \ref{thm:mean curv comp} $a=\infty$ can only
happen if $\kappa=0$ or $\kappa<0$ (for $\kappa>0$ one has $b_{\kappa,\beta}<\infty$
for any $\beta\in\mathbb{R}$) in which cases the functions behave
nicely at infinity and we set $\frac{1}{a-t}:=0$ and $-\sqrt{\left|\kappa\right|}\coth(\sqrt{\left|\kappa\right|}(a-t)):=-\sqrt{\left|\kappa\right|}$,
respectively.\end{prop}
\begin{proof}
The proof is completely analogous to \citep[Prop. 4.9, 4.]{AGH1997}.
Given any $p\in N_{t}$ there exists a timelike asymptote $\alpha_{p}:[0,a-t)\to M$
by Prop.\ \ref{prop:properties of buseman}. Now we look at $S_{\alpha_{p}(s)}^{-}(s):=\{x\in M:\tau_{x}(\alpha_{p}(s))=s\}$.
Clearly $S_{\alpha_{p}(s)}^{-}(s)$ is a smooth hypersurface for any
$s\in(0,a-t)$, $p\in S_{\alpha_{p}(s)}^{-}(s)$, and by Thm.\ \ref{thm:box of distance funct}
\[
H_{S_{\alpha_{p}(s)}^{-}(s)}\geq-(n-1)s_{\kappa}(s).
\]
From (\ref{eq:b increasing}) and (\ref{eq:b along asymptote}) we
get immediately that $b(x)\leq b(\alpha_{p}(s))-\tau_{x}(\alpha_{p}(s))=t$
for all $x\in S_{\alpha_{p}(s)}^{-}(s)$ and invoking (\ref{eq:b increasing})
again this shows $S_{\alpha_{p}(s)}^{-}(s)\cap I^{+}(N_{t})=\emptyset$.
Since $N_{t}$ is an acausal topological hypersurface in $U$ its
Cauchy development $D$ (in $U$) must be open (\citep[Lem.~14.43]{ONeill_SRG}),
$N_{t}$ is edgeless and acausal in $D$ and $S_{\alpha_{p}(s)}^{-}(s)\cap D\subset J^{-}(N_{t},D)$
(because as noted above $S_{\alpha_{p}(s)}^{-}(s)\cap I^{+}(N_{t})=\emptyset$),
hence the $S_{\alpha_{p}(s)}^{-}(s)$ lie locally to the past of $N_{t}$
near $p$ for any $s\in(0,a-t)$. But this means that they are past
support hypersurfaces with the right curvature bounds. By (\ref{eq:blubbbbbbbb})
the unit normals $\alpha_{p}'(0)$ are contained in a compact set
for all $p\in N_{t}\cap U$, so we can use \citep[Prop.~3.5]{AGH1997}
to see hat they also satisfy the estimate (\ref{eq:sec fund form hessian bounds})
on the second fundamental form.
\end{proof}
Combining the above with the mean curvature comparison Thm.\ \ref{thm:mean curv comp}
and the geometric maximum principle (Thm.\ \ref{thm:geom max princ},
\citep[Prop.~4.6]{AGH1997}) yields the following analogue to \citep[Lem.~3.2]{GALLOWAY1989_4}:
\begin{prop}
\label{prop:ein ray kommt selten allein}Assume that $\left(M,g,\Sigma\right)$
satisfies $CCC(\kappa,\beta)$ with either $\kappa>0$ or $\beta\leq-(n-1)\sqrt{\left|\kappa\right|}$.
If $\gamma:[0,b_{\kappa,\beta})\to M$ is a maximal $\Sigma$-ray,
then there exists a neighborhood $U$ of $\gamma(0)$ in $\Sigma$
such that any inextendible (f.d., unit-speed) geodesic $\sigma$ with
$\sigma'(0)\in T\Sigma^{\perp}|_{U}$ is also a $\Sigma$-ray.\end{prop}
\begin{proof}
Choose a neighborhood $V$ of $\gamma(0)$ in $\Sigma$ and $\delta>0$
small enough such that $\exp^{N}$ is smooth on $V\times(-\delta,\delta)$
and denote by $\Sigma_{\delta}$ the hypersurface $\exp^{N}(V\times\{\delta\})$.
Note that by shrinking $V$ if necessary we can assume $\Sigma_{\delta}\subset U$
for a nice neighborhood $U$ of $\gamma(\delta)$. From (\ref{eq:busem geq dist})
it follows that $b(p)\geq\delta$ for all $p\in\Sigma_{\delta}$.
Since $b$ is strictly increasing along timelike curves (again (\ref{eq:b increasing}))
this shows $\Sigma_{\delta}\cap I^{-}(N_{\delta})=\emptyset$, so
as before looking at the Cachy development $D$ of $U\cap N_{\delta}$
we see that $\Sigma_{\delta}\subset J^{+}(N_{\delta},D)$ and obviously
$\gamma(\delta)\in\Sigma_{\delta}\cap N_{\delta}$, so $\Sigma_{\delta}$
lies locally to the future of $N_{\delta}$. Now Thm.\ \ref{thm:mean curv comp}
shows $H_{\Sigma_{\delta}}\leq H_{\kappa,\beta}(\delta)$ and comparing
the definition of $s_{\kappa}$ in (\ref{eq:def s kappa}) with Table
\ref{tab:Warping-functions-for} shows $H_{\kappa,\beta}(\delta)=-(n-1)s_{\kappa}(b_{\kappa,\beta}-\delta)$
if $b_{\kappa,\beta}<\infty$ (i.e., $\kappa>0$ or $\beta<-(n-1)\sqrt{\left|\kappa\right|}$)
and $H_{\Sigma_{\delta}}\leq H_{\kappa,\beta}(\delta)\leq\beta=-(n-1)\sqrt{\left|\kappa\right|}=\lim_{r\to\infty}-(n-1)s_{\kappa}(r)$
if $b_{\kappa,\beta}=\infty$ (i.e., $\beta=-(n-1)\sqrt{\left|\kappa\right|}$).
Thus, taking into acount the lower bound (\ref{eq:lower bound on H N})
on $H_{N_{\delta}}$) we can apply Thm.\ \ref{thm:geom max princ}
to obtain $N_{\delta}=H_{\Sigma_{\delta}}$. 

Now for any $p\in V$ we look at the curve $\tilde{\alpha}_{p}:[0,b_{\kappa,\beta})\to M$
given by
\[
\tilde{\alpha}_{p}(t):=\begin{cases}
\exp^{N}(t\mathbf{n}_{p}) & 0\leq t\leq\delta\\
\alpha_{\exp^{N}(\delta\mathbf{n}_{p})}(t-\delta) & \delta\leq t<b_{\kappa,\beta}
\end{cases}.
\]
This curve satisfies $\tau_{\Sigma}(\tilde{\alpha}_{p}(t))=t$: By
(\ref{eq:b along asymptote}) one has $b(\tilde{\alpha}_{p}(t))=t-\delta+b(\exp^{N}(\delta\mathbf{n}_{p}))=t$
and the claim follows from (\ref{eq:busem geq dist}). Because $\tilde{\alpha}_{p}$
is parametrized by arc-length this shows that $\tilde{\alpha}_{p}$
always maximizes the distance to $\Sigma$ so it has to be a geodesic
starting orthogonally to $\Sigma$ and a $\Sigma$-ray.
\end{proof}
The previous result allows us to prove a local splitting via Thm.\
\ref{thm: max injectivity radius}. To extend this to a global one
we need one more Lemma.
\begin{lem}
\label{lem:I- gamma ist alles for comparison}Let $\kappa,\beta\in\mathbb{R}$
with either $\kappa>0$ or $\beta\leq-(n-1)\sqrt{\left|\kappa\right|}$,
let $(\Sigma,h)$ be an $(n-1)$-dimensional Riemannian manifold,
and let $M:=[0,b_{\kappa,\beta})\times_{f_{\kappa,\beta}}\Sigma$.
Then for any $t\in[0,b_{\kappa,\beta})$ and any $r>0$ there exists
$\tilde{t}\in(t,b_{\kappa,\beta})$ such that $\{t\}\times B_{p}(r)\subset J^{-}((\tilde{t},p))$
for all $p\in\Sigma$. Furthermore $M=J^{-}([0,b_{\kappa,\beta})\times\{p\})$
for any $p\in\Sigma$.\end{lem}
\begin{proof}
We look at (future directed) null curves $c=(c_{0},\bar{c})$ starting
at a point $(t,p)$ such that the projection $\bar{c}$ is a unit-speed
curve in $(\Sigma,h)$. This yields the ODE $c_{0}'(s)^{2}=f_{\kappa,\beta}^{2}(c_{0}(s))\left|\bar{c}'(s)\right|_{h}^{2}=f_{\kappa,\beta}^{2}(c_{0}(s))$
with $c_{0}(0)=t$. Since we want $c$ to be future directed, we need
$c'_{0}>0$, so the ODE becomes $c_{0}'(s)=\left|f_{\kappa,\beta}(c_{0}(s))\right|$.
Noting that $f_{\kappa,\beta}^{2}$ is monotonously decreasing on
$[r_{\kappa,\beta},b_{\kappa,\beta})$ for some $r_{\kappa,\beta}<b_{\kappa,\beta}$(see
Table \ref{tab:Warping-functions-for}) this gives that $c_{0}(s)\leq\left|f_{\kappa,\beta}(t)\right|s+t$
for $t\geq r_{\kappa,\beta}$. So given any radius $r$ there exists
$t$ such that $c_{0}(r)\leq\left|f_{\kappa,\beta}(t)\right|r+t<b_{\kappa,\beta}$
Now let $p,q\in\Sigma$ with $q\in S_{p}(\bar{r})$ for $\bar{r}<r$
then there is a future directed null curve $c:[0,\tilde{r}]\to M$
(with $r>\tilde{r}\geq\bar{r}$ since there may not exist a curve
from $p$ to $q$ in $\Sigma$ of minimal length) from $(t,q)$ to
$(c_{0}(\tilde{r}),p)$, i.e., $(t,q)\in J^{-}((c_{0}(\tilde{r}),p))\subset J^{-}((c_{0}(r),p))$.
This finishes the proof.
\end{proof}
Now we are ready to prove the theorem. 
\begin{thm}
\label{thm:maximal ray splitting}Assume that $\left(M,g,\Sigma\right)$
satisfies $CCC(\kappa,\beta)$ with constants $\kappa,\beta$ such
that $\kappa>0$ or $\beta\leq-(n-1)\sqrt{\left|\kappa\right|}$.
If $M$ contains a maximal $\Sigma$-ray $\gamma:[0,b_{\kappa,\beta})\to M$,
then $I^{+}(\Sigma)$ is isometric to the warped product 
\begin{equation}
I^{+}(\Sigma)\cong(0,b_{\kappa,\beta})\times_{f_{\kappa,\beta}}(\Sigma,\frac{1}{f_{\kappa,\beta}(0)^{2}}\, g|_{\Sigma}).\label{eq:I+ isom max ray-1}
\end{equation}
\end{thm}
\begin{proof}
Let $U\subset\Sigma$ be as in Prop.\ \ref{prop:ein ray kommt selten allein}
and let $j:\mathbb{R}\times\Sigma\to T\Sigma^{\perp}$denote the map
$(t,p)\mapsto t\mathbf{n}_{p}$. Then $\exp^{N}\circ j:(0,b_{\kappa,\beta})\times U\to M$
is a diffeomorphism onto its image and by Thm.\ \ref{thm: max injectivity radius}
even an isometry if we equip $(0,b_{\kappa,\beta})\times U$ with
the metric $-dt^{2}+\frac{f_{\kappa,\beta}(t)^{2}}{f_{\kappa,\beta}(0)^{2}}g|_{U}$.
Now let $r>0$ be such that $U=B_{r}(\gamma(0))$ is the largest open
ball in $\Sigma$ such that $\exp^{N}\circ j|_{(0,b_{\kappa,\beta})\times U}$
is a diffeomorphism. If $U=\Sigma$ we are done. Otherwise there exists
a point $p\in\partial U$ such that $t\mapsto\exp^{N}(t\mathbf{n}_{p})=:\sigma(t)$
either stops existing or being maximizing before $b_{\kappa,\beta}$.

If it stops being maximizing but not existing the cut function $s_{\Sigma}^{+}:S^{+}N\Sigma\to(0,\infty]$
is continuous at $\dot{\sigma}(0)$ by Lem.\ \ref{lem:The-cut-function is continuous},
so we find a neighborhood $V$ of $p$ such that all f.d., unit-speed
geodesics starting in $V$ orthogonally to $\Sigma$ also have a cut
parameter less than $b_{\kappa,\beta}$, which contradicts $p\in\partial U$. 

If it stops existing at $T<b_{\kappa,\beta}$, then $\sigma\subset\overline{\exp^{N}([0,T)\cdot S^{+}N\Sigma|_{U})}$.
Now by Lem.\ \ref{lem:I- gamma ist alles for comparison} there exists
$\tilde{t}<b_{\kappa,\beta}$ such that $\{T\}\times B_{r}(\gamma(0))\subset J^{-}((\tilde{t},\gamma(0)))$,
hence $[0,T]\times B_{r}(\gamma(0))\subset J^{-}((\tilde{t},\gamma(0)))$.
But this shows $\sigma\subset\overline{\exp^{N}([0,T)\cdot S^{+}N\Sigma|_{U})}\subset J^{-}(\gamma(\tilde{t}))$,
so $\sigma$ is contained in the compact set $J^{+}(p)\cap J^{-}(\gamma(\tilde{t}))$,
contradicting its inextendibility.
\end{proof}

\section{\label{sec:Maximal-volume-rigidity}A splitting theorem for maximal
volume}

In this section we are going to look at spacetimes that are in a sense
maximal in volume, specifically we want the volume of distance balls
$B_{A}^{+}(t)$ over a set $A\subset\Sigma$ to be maximal. Obviously
this volume depends on the area of the base set, so we first introduce
a function $v_{\kappa,\beta}$ on our comparison spaces giving the
volume of future balls over a subset $A\subset\Sigma_{\kappa,\beta}$
in $M_{\kappa,\beta}$ relative to the area of $A$. 
\begin{defn}
Given $\kappa,\beta\in\mathbb{R}$ and any measurable set $A\subset\Sigma_{\kappa,\beta}$
with non-zero measure we define 
\[
v_{\kappa,\beta}(t):=\frac{\mathrm{vol}_{\kappa,\beta}B_{A}^{+}(t)}{\mathrm{area}_{\kappa,\beta}A}.
\]
Note that $\mathrm{tr}\,\mathbf{S}_{\kappa,\beta}=H_{\kappa,\beta}=(n-1)\frac{f'}{f}$
for warped products (\citep[Prop.~7.35]{ONeill_SRG}), so the variation
of area formula (\ref{eq:var of area}) and the coarea formula (\ref{eq:coarea formula})
show 
\[
v_{\kappa,\beta}(t)=\frac{1}{f_{\kappa,\beta}(0)^{n-1}}\int_{0}^{t}f_{\kappa,\beta}(\tau)^{n-1}d\tau
\]
for all $t\leq b_{\kappa,\beta}$. For $t\geq b_{\kappa,\beta}$ one
obviously has $v_{\kappa,\beta}(t)=v_{\kappa,\beta}(b_{\kappa,\beta})=:\bar{v}_{\kappa,\beta}$,
so $v_{\kappa,\beta}$ really is independent of the choice of $A$.
If $\kappa>0$ or $\beta<-(n-1)\sqrt{\left|\kappa\right|}$ then $b_{\kappa,\beta}<\infty$
and hence $\bar{v}_{\kappa,\beta}<\infty$. 

We are now ready to prove a splitting theorem if the volume of $B_{K}^{+}:=\bigcup_{s\in(0,\infty)}S_{K}^{+}(s)$
is finite and maximal (w.r.t.\ (\ref{eq:vol leq sth})) for compact
$K\subset\Sigma$.\end{defn}
\begin{thm}
[Maximal volume splitting] \label{thm:- max vol rigidity}If $\left(M,g,\Sigma\right)$
satisfies $CCC(\kappa,\beta)$ with either $\kappa>0$ or $\beta<-(n-1)\sqrt{\left|\kappa\right|}$
and there exists an exhaustion by compact sets $\{K_{n}\}_{n\in\mathbb{N}}$
for $\Sigma$ such that
\begin{equation}
\frac{\mathrm{vol}B_{K_{n}}^{+}}{\mathrm{area}K_{n}}=\bar{v}_{\kappa,\beta}\label{eq:vol max}
\end{equation}
for all $n$, then
\begin{equation}
I^{+}(\Sigma)\cong(0,b_{\kappa,\beta})\times_{f_{\kappa,\beta}}(\Sigma,\frac{1}{f_{\kappa,\beta}(0)^{2}}\, g|_{\Sigma}).\label{eq:max vol isometry}
\end{equation}

If furthermore $\Sigma$ is PCC, $\kappa>0$ and there exists an exhaustion
of compact sets $\{K_{n}\}_{n\in\mathbb{N}}$ for $\Sigma$ such that
also
\[
\frac{\mathrm{vol}B_{K_{n}}^{-}}{\mathrm{area}K_{n}}=\bar{v}_{\kappa,-\beta}
\]
for all $n$, then
\begin{equation}
M\cong(a_{\kappa,\beta},b_{\kappa,\beta})\times_{f_{\kappa,\beta}}(\Sigma,\frac{1}{f_{\kappa,\beta}(0)^{2}}\, g|_{\Sigma}).\label{eq:M isom maxvol}
\end{equation}
\end{thm}
\begin{proof}
By Cor.\ \ref{cor: tau<b on I+} we have $\tau_{\Sigma}(q)<b_{\kappa,\beta}$
for all $q\in I^{+}(\Sigma)$, so $\mathrm{vol}B_{K_{n}}^{+}(b_{\kappa,\beta})=\mathrm{vol}B_{K_{n}}^{+}=\bar{v}_{\kappa,\beta}\,\mathrm{area}K_{n}=v_{\kappa,\beta}(b_{\kappa,\beta})\,\mathrm{area}K_{n}$.
Using this and the coarea formula (\ref{eq:coarea formula}) it follows
that
\begin{equation}
0=v_{\kappa,\beta}(b_{\kappa,\beta})-\frac{\mathrm{vol}B_{K_{n}}^{+}(b_{\kappa,\beta})}{\mathrm{area}K_{n}}=\int_{0}^{b_{\kappa,\beta}}\frac{\mathrm{area}_{\kappa,\beta}S_{A}^{+}(\tau)}{\mathrm{area}_{\kappa,\beta}A}-\frac{\mathrm{area}\EuScript S_{K_{n}}^{+}(\tau)}{\mathrm{area}K_{n}}d\tau\label{eq:whatevs}
\end{equation}
for any $A\subset\Sigma_{\kappa,\beta}$ with finite, non-zero measure.
Now by the area comparison theorem Thm.\ \ref{thm:area comp} the
integrand is always non-negative, so it has to be zero almost everywhere
and we obtain
\begin{equation}
a_{\kappa,\beta}(t):=\frac{\mathrm{area}_{\kappa,\beta}S_{A}^{+}(t)}{\mathrm{area}_{\kappa,\beta}A}=\frac{\mathrm{area}\EuScript S_{K_{n}}^{+}(t)}{\mathrm{area}K_{n}}\label{eq:area max}
\end{equation}
for almost all $t\leq b_{\kappa,\beta}$. Since $t\mapsto\frac{\mathrm{area}\EuScript S_{K_{n}}^{+}(t)}{\mathrm{area}_{\kappa,\beta}S_{A}^{+}(t)}$
is non-increasing (again Thm.\ \ref{thm:area comp}), the equality
(\ref{eq:area max}) follows for all $t<b_{\kappa,\beta}$. 

Next we show that thus (\ref{eq:area max}) holds for any compact
set $K\subset\Sigma$: Given $K$ choose $n\in\mathbb{N}$ such that
$K\subset K_{n}$. Then it follows immediately from the definition
of these spheres (see Def.\ \ref{def: spheres and balls}) that $\EuScript S_{K_{n}}^{+}(t)=\EuScript S_{K}^{+}(t)\cup\EuScript S_{K_{n}\setminus K}^{+}(t)$
and the union is disjoint. But then
\begin{multline*}
a_{\kappa,\beta}(t)\,\mathrm{area}K_{n}=\mathrm{area}\EuScript S_{K_{n}}^{+}(t)=\mathrm{area}\EuScript S_{K}^{+}(t)+\mathrm{area}\EuScript S_{K_{n}\setminus K}^{+}(t)\leq a_{\kappa,\beta}(t)\,\mathrm{area}K+\mathrm{area}\EuScript S_{K_{n}\setminus K}^{+}(t)\leq\\
\leq a_{\kappa,\beta}(t)\,(\mathrm{area}K+\mathrm{area}K_{n}\setminus K)=a_{\kappa,\beta}(t)\,\mathrm{area}K_{n},
\end{multline*}
so all inequalities have to be equalities, showing (\ref{eq:area max})
for $K$.

This allows us to prove that actually $\mathrm{Cut}^{+}(\Sigma)=\emptyset$:
First note that it suffices to show $\mathrm{Cut}^{+}(\Sigma)\cap S_{\Sigma}^{+}(t)=\emptyset$
for all $t<b_{\kappa,\beta}$ since we know $\tau_{\Sigma}(q)<b_{\kappa,\beta}$
for all $q\in I^{+}(\Sigma)$ from Cor.\ \ref{cor: tau<b on I+}.
Assume $p\in\mathrm{Cut}^{+}(\Sigma)\cap S_{\Sigma}^{+}(t)$, then
$p=\gamma(t)$ for some f.d., unit-speed geodesic $\gamma$ starting
orthogonally to $\Sigma$ with $t=s_{\Sigma}^{+}(\dot{\gamma}(0))$.
Since this $\gamma$ is certainly defined on an open interval containing
$t$ the cut function $s_{\Sigma}^{+}:S^{+}N\Sigma\to(0,\infty]$
is continuous at $\dot{\gamma}(0)$ by Lem.\ \ref{lem:The-cut-function is continuous}.
Let $\varepsilon>0$ with $t+\varepsilon<b_{\kappa,\beta}$. We can
choose a relatively compact neighborhood $V$ in $\Sigma$ of $\gamma(0)$
such that all f.d., unit-speed geodesics starting in $\bar{V}$ orthogonally
to $\Sigma$ have a cut parameter less than $t+\varepsilon<b_{\kappa,\beta}$.
But then $\EuScript S_{\bar{V}}^{+}(t+\varepsilon)=\emptyset$, contradicting
$0\neq a_{\kappa,\beta}(t+\varepsilon)=\frac{\mathrm{area}\EuScript S_{\bar{V}}^{+}(t)}{\mathrm{area}\bar{V}}$.

Next we will show that $H_{t}(q)=H_{\kappa,\beta}(t)$ for all $q\in I^{+}(\Sigma)$
with $t=\tau_{\Sigma}(q)$. To see this, let $\gamma$ be the unique
geodesic from $\Sigma$ to $q$ realizing the distance and choose
$K\subset\Sigma$ to be a compact neighborhood of $\gamma(0)$ such
that the normal exponential map is defined on $[0,t')\times K$ for
some $t'>t$. By (\ref{eq:area max}) the map $t\mapsto\frac{\mathrm{area}\EuScript S_{K}^{+}(t)}{\mathrm{area}_{\kappa,\beta}S_{A}^{+}(t)}$
is constant on $[0,t')$ and since the set $\EuScript S_{K}^{+}(t)=S_{K}^{+}(t)=\exp^{N}(\{t\}\times K)$
is compact we may proceed as in the proof of the area comparison theorem
and use the first variation of area (Prop.\ \ref{prop:first var of area})
to obtain
\begin{multline*}
0=\left.\frac{d}{ds}\right|_{s=t}\log\frac{\mathrm{area}_{\kappa,\beta}S_{A}^{+}(s)}{\mathrm{area}\, S_{K}^{+}(s)}=\left.\frac{d}{ds}\right|_{s=t}\log\mathrm{area}_{\kappa,\beta}S_{A}^{+}(s)-\left.\frac{d}{ds}\right|_{s=t}\log\mathrm{area}\, S_{K}^{+}(s)=\\
=\frac{1}{\mathrm{area}\, S_{K}^{+}(t)}\int_{S_{K}^{+}(t)}H_{\kappa,\beta}(t)-H_{t}(q)d\mu_{t}(q).
\end{multline*}
Now the integrand is non-negative (by the mean curvature comparison
theorem, see Thm.\ \ref{thm:mean curv comp}) and smooth (in $q$)
on $\EuScript S_{\Sigma}^{+}(t)=S_{\Sigma}^{+}(t)$ (because the normal
exponential map is a diffeomorphism away from the cut locus), hence
$H_{t}(q)=H_{\kappa,\beta}(t)$ for all $q\in S_{K}^{+}(t)$. 

By Thm.\ \ref{thm:mean curv comp} this already implies $\mathbf{S}_{t}=H_{\kappa,\beta}(t)\,\mathrm{id}=\frac{f'_{\kappa,\beta}(t)}{f_{\kappa,\beta}(t)}\,\mathrm{id}$
for all $t<b_{\kappa,\beta}$. Unfortunately, $\exp^{N}$ need a priori
not be defined on all of $(0,b_{\kappa,\beta})\cdot S^{+}N\Sigma$,
so there is still some more work to do than in Thm.\ \ref{thm: max injectivity radius}.
We can, however, proceed similarly: Using the normal exponential map
we obtain coordinates $(t,x)$ on an open submanifold of $M$ containing
$I^{+}(\Sigma)\cup\Sigma$ (note again that $\mathrm{Cut}^{+}(\Sigma)=\emptyset$)
in which $g=-dt^{2}+h(t,x)$ where for any $0\leq t<b_{\kappa,\beta}$
the expression $h(t,.)$ denotes the induced Riemannian metric on
the spacelike hypersurface $\EuScript S_{\Sigma}^{+}(t)=S_{\Sigma}^{+}(t)$
(which is just the $\{t\}$-level set of the distance function $\tau_{\Sigma}$).
Calculating as in (\ref{eq:curv calc}) we see that for $t>0$ and
$x\in S_{\Sigma}^{+}(t)$ 
\[
\frac{d}{dt}h_{ij}(t,x)=2\frac{f'_{\kappa,\beta}(t)}{f_{\kappa,\beta}(t)}h_{ij}(t,x).
\]
The solution of this equation is again given by $h_{ij}(t,x)=\frac{h_{ij}(0,x)}{f_{\kappa,\beta}(0)^{2}}\, f_{\kappa,\beta}(t)^{2}$.

This shows that \[ I^{+}(\Sigma,M)\cong(\mathcal{D}\cap\left((0,b_{\kappa,\beta})\cdot S^{+}N\Sigma\right),-dt^{2}+\frac{f_{\kappa,\beta}(t)^{2}}{f_{\kappa,\beta}(0)^{2}}g|_{\Sigma})\subset[0,b_{\kappa,\beta})\times_{f_{\kappa,\beta}}(\Sigma,\frac{1}{f_{\kappa,\beta}(0)^{2}}\, g|_{\Sigma})=:M_{\kappa,\beta}, \]
so it is isometric to an open submanifold of the warped product.

It only remains to show that all f.d., unit-speed geodesics starting
orthogonally to $\Sigma$ are defined in $M$ on $[0,b_{\kappa,\beta})$,
i.e., they remain in the submanifold $I^{+}(\Sigma,M)\subset M_{\kappa,\beta}$.
Assume to the contrary that there exists such a geodesic $\gamma:[0,T)\to I^{+}(\Sigma,M)$
with $T<b_{\kappa,\beta}$ that is inextendible in $M$. Let $\varepsilon>0$
such that $T+\varepsilon<b_{\kappa,\beta}$. Then $\mathrm{area}S_{\bar{U}}^{+}(T+\varepsilon)$
has to be maximal for any relatively compact neighborhood $U$ of
$q:=\gamma(0)$ in $\Sigma$ and hence non-zero, in particular $S_{\bar{U}}^{+}(T+\varepsilon)\neq\emptyset$.
Thus there exists a sequence of $q_{n}\in\Sigma$ with $d_{\Sigma}(q,q_{n})=\frac{1}{n}$
(where $d_{\Sigma}$ is the Riemannian distance on $\Sigma$ induced
by $g|_{\Sigma}$) such that the corresponding $\gamma_{n}$ exist
until at least $T+\varepsilon$. Set $p_{n}:=\gamma_{n}(T+\varepsilon)$. 

Let $V$ be a relatively compact and geodesically convex neighborhood
of $q$ in $\Sigma$ and choose $N$ such that that $q_{n}\in V$
for all $n>N$. Now for any $0<\delta\leq T$ let $\sigma_{n,\delta}:[0,s_{\mathrm{max}})\to M_{\kappa,\beta}$
be defined by $\sigma_{n,\delta}(s):=(T+\varepsilon-s,c_{n}(s\,\frac{1}{n(\varepsilon+\delta)}))$
where $c_{n}:[0,b)\to\Sigma$ is the \emph{unit-speed} geodesic in
$\Sigma$ starting at $q_{n}$ in direction $q$. Note that because
$V$ was chosen to be geodesically convex the curve $\sigma_{n,\delta}$
is actually well-defined on $[0,\varepsilon+\delta]$, its projection
to the second coordinate is contained in $V$ and $\sigma_{n,\delta}(0)=(T+\varepsilon,c_{n}(0))=\gamma_{n}(T+\varepsilon)=p_{n}$
and $\sigma_{n,\delta}(\varepsilon+\delta)=(T-\delta,c_{n}(d_{\Sigma}(q,q_{n})))=(T-\delta,q)=\gamma(T-\delta)$.
We have $\dot{\sigma}_{n,\delta}(s)=(-1,\frac{1}{n(\varepsilon+\delta)}\dot{c}_{n}(\frac{s}{n(\varepsilon+\delta)}))$,
so for $n>\max_{s\in[0,T+\varepsilon]}\frac{f_{\kappa,\beta}(T+\varepsilon-s)^{2}}{f_{\kappa,\beta}(0)^{2}\varepsilon}$
we have 
\[
g(\dot{\sigma}_{n,\delta}(s),\dot{\sigma}_{n,\delta}(s))=-1+\frac{f_{\kappa,\beta}(T+\varepsilon-s)^{2}}{f_{\kappa,\beta}(0)^{2}n(\varepsilon+\delta)}<0
\]
Note that this bound on $n$ is independent of $\delta$. So if we
fix $N$ large enough, we see that, at least in $M_{\kappa,\beta}$,
the curves $\sigma_{N,\delta}:[0,\varepsilon+\delta]\to M_{\kappa,\beta}$
can be used to give a timelike connection from $p_{N}$ to any point
on $\gamma$. 

Next we show that actually $\sigma_{N,\delta}:[0,\varepsilon+\delta]\to I^{+}(\Sigma,M)\subset M\subset M_{\kappa,\beta}$
for any $0<\delta<T$. Fix $\delta$. Since we chose $\sigma_{N,\delta}(0)=p_{N}\in I^{+}(\Sigma,M)$
and $I^{+}(\Sigma,M)\subset M_{\kappa,\beta}$ is open we get that
$s_{0}:=\sup\{s\in[0,\varepsilon+\delta]:\,\sigma_{N,\delta}|_{[0,s)}\subset I^{+}(\Sigma,M)\}>0$.
If $s_{0}=\varepsilon+\delta$ we are finished since then $\sigma_{N,\delta}=\sigma_{N,\delta}|_{[0,\varepsilon+\delta)}\cup\gamma(T-\delta)\subset I^{+}(\Sigma,M)$.
So assume that $0<s_{0}<\varepsilon+\delta$. Then the curve $\sigma_{N,\delta}|_{[0,s_{0})}\subset I^{+}(\Sigma,M)$
is a (past) inextendible, p.d., timelike curve in $M$ and  $\sigma_{N,\delta}([0,s_{0}))\subset J^{-}(p_{N},M)\cap J^{+}(\Sigma,M)$
which is compact (because $\Sigma$ is FCC and $M$ is globally hyperbolic).
This contradicts global hyperbolicity of $M$.

This shows that $\sigma_{N,\delta}$ is a timelike curve from $p_{N}$
to $\gamma(T-\delta)$ in $M$ for any $0<\delta<T$, so the original
inextendible geodesic $\gamma$ is contained in $J^{-}(p_{N},M)\cap J^{+}(q,M)$,
which again contradicts global hyperbolicity of $M$.

The second assertion follows by reversing the time orientation of
$M$ (note that while a bound from above on $H_{\Sigma}^{+}$ will
in general only translate to a bound from below for $H_{\Sigma}^{-}$,
the previous calculations show that $H_{\Sigma}^{+}$ and hence also
$H_{\Sigma}^{-}$ are constant anyways).
\end{proof}
Contrary to the earlier two results (Thm.\ \ref{thm: max injectivity radius}
and Thm.\ \ref{thm:maximal ray splitting}) this last theorem can
easily be adapted to all remaining possible values of $\kappa,\beta$
(and not only $\kappa\leq0$ and $\beta=-(n-1)\sqrt{\left|\kappa\right|}$),
by slightly tweaking the assumptions.
\begin{prop}
\label{prop:Let--satisfy}Let $\left(M,g,\Sigma\right)$ satisfy $CCC(\kappa,\beta)$
with $\kappa\leq0$ and $\beta>-(n-1)\sqrt{\left|\kappa\right|}$
and assume that there exists an exhaustion by compact sets $\{K_{m}\}_{m\in\mathbb{N}}$
for $\Sigma$ and a sequence of times $t_{n}\to\infty$ such that
\[
\lim_{n\to\infty}\left(v_{\kappa,\beta}(t_{n})-\frac{\mathrm{vol}B_{K_{m}}^{+}(t_{n})}{\mathrm{area}K_{m}}\right)=0
\]
for all $m$, then
\[
I^{+}(\Sigma)\cong(0,b_{\kappa,\beta})\times_{f_{\kappa,\beta}}(\Sigma,\frac{1}{f_{\kappa,\beta}(0)^{2}}\, g|_{\Sigma}).
\]
\end{prop}
\begin{proof}
The proof remains largely same, only in (\ref{eq:whatevs}) one uses
that
\begin{multline*}
0=\lim_{n\to\infty}\left(v_{\kappa,\beta}(t_{n})-\frac{\mathrm{vol}B_{K_{m}}^{+}(t_{n})}{\mathrm{area}K_{m}}\right)=\lim_{n\to\infty}\int_{0}^{t_{n}}\frac{\mathrm{area}_{\kappa,\beta}S_{A}^{+}(\tau)}{\mathrm{area}_{\kappa,\beta}A}-\frac{\mathrm{area}\EuScript S_{K_{m}}^{+}(\tau)}{\mathrm{area}K_{m}}d\tau=\\
=\int_{0}^{\infty}\frac{\mathrm{area}_{\kappa,\beta}S_{A}^{+}(\tau)}{\mathrm{area}_{\kappa,\beta}A}-\frac{\mathrm{area}\EuScript S_{K_{m}}^{+}(\tau)}{\mathrm{area}K_{m}}d\tau
\end{multline*}
by positivity of the integrand to get (\ref{eq:area max}) for almost
all $t<\infty$. The rest follows exactly as above.
\end{proof}
To summarize, the above Thm.\ \ref{thm:- max vol rigidity} and Prop.\
\ref{prop:Let--satisfy} complement the main splitting Theorem \ref{thm:maximal ray splitting}
nicely: Using a slightly stronger assumption leads to both a very
natural and elementary proof and a natural generalization to all possible
curvature bounds (whereas Thm.\ \ref{thm:maximal ray splitting}
only looks at ones that lead to a finite bound $b_{\kappa,\beta}$
on $\tau_{\Sigma}$ or that are boundary cases in the sense that $b_{\kappa,\beta}=\infty$
but $b_{\kappa,\bar{\beta}}<\infty$ for all $\bar{\beta}<\beta$).

\subsection*{Acknowledgment}

The author is the recipient of a DOC Fellowship of the Austrian Academy
of Sciences at the Institute of Mathematics at the University of Vienna.
This work was also partially supported by the Austrian Science Fund
(FWF) project number P28770. 

\bibliographystyle{amsalpha}
\bibliography{bibtex_BS2}

\providecommand{\bysame}{\leavevmode\hbox to3em{\hrulefill}\thinspace}
\providecommand{\MR}{\relax\ifhmode\unskip\space\fi MR }
\providecommand{\MRhref}[2]{%
  \href{http://www.ams.org/mathscinet-getitem?mr=#1}{#2}
}
\providecommand{\href}[2]{#2}
\begin{thebibliography}{BEMG85}

\bibitem[AGH96]{AGH1997}
L.~Andersson, G.~J. Galloway, and R.~Howard, \emph{{A strong Maximum Principle
  for weak solutions of quasi-linear Elliptic Equations with applications to
  Lorentzian and Riemannian Geometry}}, Comm. Pure Appl. Math \textbf{51}
  (1996), 581--624.

\bibitem[AH98]{AnderssonHoward}
L.~Andersson and R.~Howard, \emph{{Comparison and rigidity theorems in
  semi-{R}iemannian geometry}}, Comm. Anal. Geom. \textbf{6} (1998), no.~4,
  819--877.

\bibitem[BEE96]{BEE96}
J.~K. Beem, P.~E. Ehrlich, and K.~L. Easley, \emph{{Global Lorentzian
  Geometry}}, Dekker, New York, 1996.

\bibitem[BEMG85]{BEMG1985}
J.~K. Beem, P.~E. Ehrlich, S.~Markvorsen, and G.~J. Galloway,
  \emph{{Decomposition theorems for Lorentzian manifolds with nonpositive
  curvature}}, J. Differential Geom. \textbf{22} (1985), no.~1, 29--42.

\bibitem[BS05]{bernalSanchez_globHypSplitting}
A.~N. Bernal and M.~S{\'a}nchez, \emph{{Smoothness of Time Functions and the
  Metric Splitting of Globally Hyperbolic Spacetimes}}, Communications in
  Mathematical Physics \textbf{257} (2005), 43--50.

\bibitem[BS06]{BernalSanchez2006}
\bysame, \emph{{Further results on the smoothability of Cauchy hypersurfaces
  and Cauchy time functions}}, Letters in Mathematical Physics \textbf{77}
  (2006), no.~2, 183--197.

\bibitem[CG71]{CGro}
J.~Cheeger and D.~Gromoll, \emph{{The splitting theorem for manifolds of
  nonnegative Ricci curvature}}, J. Diff. Geom. \textbf{6} (1971), 119--128.

\bibitem[Che75]{Cheng}
S.~Y. Cheng, \emph{{Eigenvalue comparison theorems and geometric
  applications}}, Math. Z. \textbf{143} (1975), 289--297.

\bibitem[EG92]{EG1992}
J.-H. Eschenburg and G.~J. Galloway, \emph{Lines in space-times}, Comm. Math.
  Phys. \textbf{148} (1992), no.~1, 209--216.

\bibitem[EH90]{EH}
J.-H. Eschenburg and E.~Heintze, \emph{{Comparison theory for Riccati
  equations}}, Manuscripta Math. \textbf{68} (1990), 209--214.

\bibitem[Esc88]{eschenburg1988}
J.-H. Eschenburg, \emph{The splitting theorem for space-times with strong
  energy condition}, J. Differential Geom. \textbf{27} (1988), no.~3, 477--491.

\bibitem[Gal89a]{GALLOWAY1989_4}
G.~J. Galloway, \emph{{Some connections between global hyperbolicity and
  geodesic completeness}}, Journal of Geometry and Physics \textbf{6} (1989),
  no.~1, 127 -- 141.

\bibitem[Gal89b]{galloway1989_3}
\bysame, \emph{{The Lorentzian splitting theorem without the completeness
  assumption}}, J. Differential Geom. \textbf{29} (1989), no.~2, 373--387.

\bibitem[Ger70]{Geroch1970uw}
R.~P. Geroch, \emph{{The domain of dependence}}, J. Math. Phys. \textbf{11}
  (1970), 437--439.

\bibitem[Gra16]{G2016}
M.~Graf, \emph{{Volume comparison for $C^{1,1}$ metrics}}, {Ann. Global Anal.
  Geom.} (2016), doi:10.1007/s10455--016--9508--2.

\bibitem[New90]{newman1990}
R.~P. A.~C. Newman, \emph{{A proof of the splitting conjecture of S.-T. Yau}},
  J. Differential Geom. \textbf{31} (1990), no.~1, 163--184.

\bibitem[O'N83]{ONeill_SRG}
B.~O'Neill, \emph{{Semi-Riemannian Geometry}}, Academic Press, New York, 1983.

\bibitem[Sen98]{Seno1}
J.~M.~M. Senovilla, \emph{{Singularity Theorems and their consequences}}, Gen.
  Rel. Grav. \textbf{30} (1998), no.~5, 701--848.

\bibitem[Sim83]{Simon_GeomMeasure}
L.~Simon, \emph{{Lectures on Geometric Measure Theory}}, {Proceedings of the
  Centre for Mathematical Analysis, Australian National University}, vol.~3,
  Centre for Mathematical Analysis, Australian National University, 1983.

\bibitem[TG13]{TG}
J.-H. Treude and J.~D.~E. Grant, \emph{{Volume comparison for hypersurfaces in
  Lorentzian manifolds and singularity theorems}}, Ann. Global Anal. Geom.
  \textbf{43} (2013), no.~3, 233--251.

\bibitem[Tre]{Treude_Diplomarbeit}
J.-H. Treude, \emph{{Ricci curvature comparison in Riemannian and Lorentzian
  geometry, Diploma Thesis}},
  http://www.freidok.uni-freiburg.de/volltexte/8405.

\end{thebibliography}

\end{document}